\documentclass[12pt]{article}
\usepackage{amssymb}
\usepackage{amsmath}
\usepackage[a4paper]{geometry}
\usepackage{amsfonts}
\usepackage{amsthm}
\usepackage{graphicx}
\newcommand{\beq}{\begin{equation}}
\newcommand{\eeq}{\end{equation}}
\newcommand{\bea}{\begin{aligned}}
\newcommand{\eea}{\end{aligned}}
\newcommand{\bdm}{\begin{displaymath}}
\newcommand{\edm}{\end{displaymath}}
\newcommand{\barr}{\begin{array}}
\newcommand{\earr}{\end{array}}
\newcommand{\ben}{\begin{enumerate}}
\newcommand{\een}{\end{enumerate}}
\newcommand{\bde}{\begin{description}}
\newcommand{\ede}{\end{description}}

\newtheorem{teor}{Theorem}

\newtheorem{prop}[teor]{Proposition}
\newtheorem{lem}[teor]{Lemma}  
\newtheorem{cor}[teor]{Corollary}  
\newtheorem{rem}[teor]{Remark}

\newcommand{\R}{\mathbb{R}}
\newcommand{\PP}{\mathbb{P}}

\newcommand{\N}{\mathbb{N}}

\newcommand{\E}{\mathbb{E}}

\newcommand{\noi}{\noindent}
\newcommand{\defi}{\hbox{\small{$\;\stackrel{\text{def}}{=}\;$}}}

\renewcommand{\a}{\alpha}
\newcommand{\g}{\gamma}
\newcommand{\de}{\delta}
\newcommand{\D}{\Delta}
\newcommand{\e}{\epsilon}

\newcommand{\s}{\sigma}
\newcommand{\be}{\beta}

\newcommand{\vare}{\varepsilon}

\begin{document}

\title{On a nonhierarchical version of the \\Generalized Random Energy Model. II. \\Ultrametricity.}
\author{Erwin Bolthausen\thanks{Universit\"at Z\"urich, Winterthurerstrasse 190, CH-8057 Zurich.
\texttt{eb@math.uzh.ch}} \quad Nicola Kistler\thanks{ Universit\"at Bonn,
Wegelerstr. 6, DE-53115 Bonn. \texttt{nkistler@wiener.iam.uni-bonn.de}}}
\maketitle

\begin{abstract}
We study the Gibbs measure of the nonhierarchical versions of the Generalized
Random Energy Models introduced in previous work. We prove that the
ultrametricity holds only provided some nondegeneracy conditions on the
hamiltonian are met.

\end{abstract}

\newpage 
${}$
\vspace{3cm} 
${}$
\tableofcontents

\newpage

\section{Introduction}

The study of spin glasses, a paradigm for the statistical mechanics of
disordered systems, has attracted a lot of interest ever since their
introduction in the field of condensed matter. Given the success of the Ising
model for an understanding of basic questions in statistical physics, probably
the most natural spin glass model is the Edwards-Anderson model which is a
spin model with lattice $\mathbb{Z}^{d},$ and \textit{random }nearest neighbor
interactions. Mathematically, this model remains to these days totally
untractable. The situation is much better for the Sherrington-Kirkpatrick
model (SK for short), which is of mean-field type, meaning that every spin
interacts with any other on equal footing. For the SK-model, a marvellous
theory has been introduced by Giorgio Parisi in the 1970's, cfr. for more on
this \cite{parisi}, which has been further developed by many. This is a fully
developed theory which has successfully been applied to many other problems,
for instance in combinatorial optimization, but there was no mathematically
rigorous foundation, till quite recently.

In a series of groundbreaking works by Francesco Guerra \cite{guerra} and
Michel Talagrand \cite{talagrand}, the Parisi formula for the free energy has
been proved to be correct in a class of mean field models, the SK model
included. It is however very puzzling that \textit{ultrametricity }has not
been proved, although it is at the very heart of the physics theory by Parisi
and others. A metric $d$ is called an ultrametric if it satisfies the strong
triangle inequality $d\left(  x,z\right)  \leq\max\left(  d\left(  x,y\right)
,d\left(  y,z\right)  \right)  $ for any three points. This is equivalent with
the property that two balls have either no intersection, or one is contained
in the other. What is ultrametricity in the context of spin glass theory? Take
for instance the SK-model, with spin configuration space $\Sigma_{N}=\left\{
\pm 1\right\}  ^{N},$ and the Hamiltonian%
\[
H\left(  \sigma\right)  \overset{\mathrm{def}}{=}-\frac{1}{\sqrt{N}}%
\sum_{1\leq i<j\leq N}g_{ij}\sigma_{i}\sigma_{j},
\]
where the $g$'s are i.i.d. standard Gaussians. Then a natural distance is the
$L_{2}$-distance on the Hamiltonian%
\begin{align*}
d\left(  \sigma,\sigma^{\prime}\right)    & =\left\Vert H\left(
\sigma\right)  -H\left(  \sigma^{\prime}\right)  \right\Vert _{2}\\
& =\sqrt{N}\sqrt{1-R\left(  \sigma,\sigma^{\prime}\right)  ^{2}},
\end{align*}
where $R$ is the overlap of two spin configurations $\sigma,\sigma^{\prime}:$
$R\left(  \sigma,\sigma^{\prime}\right)  \overset{\mathrm{def}}{=}N^{-1}%
\sum_{i=1}^{N}\sigma_{i}\sigma_{i}^{\prime}.$ (This is a metric on $\Sigma
_{N}$ only after identifying $\sigma$ with $-\sigma$). Evidently, $d$ is not
an ultrametric. The ultrametricity conjecture for the SK-model, unproved to
this day, states that it becomes asymptotically an ultrametric for large $N$
under the Gibbs measure. This means that if one picks (for large $N$) three
independent $\sigma,\sigma^{\prime},\sigma^{\prime\prime}$ under the Gibbs
measure, then $d\left(  \sigma,\sigma^{\prime\prime}\right)  \leq\max\left(
d\left(  \sigma,\sigma^{\prime}\right)  ,d\left(  \sigma^{\prime}%
,\sigma^{\prime\prime}\right)  \right)  ,$ up to a small error, with high
probability. A precise statement in our models is given below. However, the
ultrametricity picture in the physics theory goes much beyond this, as it
gives very precise predictions on the distribution of these overlaps.
Ultrametricity was very important in the development of the Parisi theory.
First, it appeared in a somewhat hidden way in the original replica
computation, where the variational formula found by Sherrington and
Kirkpatrick was solved by Parisi using an ultrametric ansatz. Later, and
alternative \textquotedblleft cavity\textquotedblright\ approach, avoiding the
(for mathematicians horrible) replica computation, was found by M\'{e}zard,
Parisi and Virasoro, but it also relies on a hierarchical ansatz. We cannot
give even a sketch of these developments, we only want the emphasize how
important this ultrametricity picture is in spin glass theory. For details,
see \cite{parisi}.

The situation is much better in the case of the Generalized Random Energy
Model (GREM for short) introduced by Bernard Derrida in the 1980's
\cite{derrida1} for which the full Parisi picture has been proved by Bovier
and Kurkova \cite{bk2}. The GREM is however hierarchically organized from the
start, so that one gets little information on the \textit{origin} of
ultrametricity. \newline

To provide some modest insights into this issue, we introduced in \cite{bokis}
a natural nonhierarchical generalization of the GREM, for which we proved that
the limiting free energy always coincides with that of a suitably constructed
GREM, thereby getting some evidence for the validity of the ultrametricity. In
this present work we address the more difficult problem of the Gibbs measure,
and prove that the ultrametricity indeed holds, but only if some additional
assumptions on the hamiltonian are met.\newline

The problem of ultrametricity has also been addressed in several other papers,
recently. A very interesting result is by Michael Aizenman and Louis-Pierre
Arguin in \cite{aizenman_arguin} who prove that if a point process equiped
with an abstract overlap structure has a certain stability property under the
\textit{cavity dynamics} (see \cite{ass} for more on this subject), then the
overlap structure has to be hierarchical.\newline

The study of spin glasses leads to new and interesting results in probability
theory. The Gibbs measure at low temperature is evidently associated with the
minima of the Hamiltonian. In the case of spin glasses, the Hamiltonian is a
field of random variables, in the SK-case, a Gaussian field. The study of
extrema of random fields is a classical problem in probability theory. For
instance the extremal process of $n$ independent and identically distributed
random variables (under some mild assumptions on the moments) converges for
$n\rightarrow\infty$ to a Poisson point process with a certain density. As a
byproduct of our analysis, we prove that the extremal process of highly
correlated gaussian random variables such as the energy levels of our
nonhierarchical GREMs always coincides with that of a corresponding
hierarchical field, cfr. Corollary \ref{extremal_process_theorem}.

\section{Nonhierarchical GREM and ultrametricity} 
We recall the construction of the non hierarchical GREMs. Throughout this paper, we fix a number $n \in \N$, and consider the set
$I = \{1, . . . ,n\}$, as well as a collection of positive real numbers $(a_J, J\subset I)$ such
that $\sum_{J\subset I} a_J = 1$. For convenience, we put $a_\emptyset
\defi = 0$. The relevant subset of $I$ will be only the
ones with positive a-value. For $A \subset I$, we set
\[\mathcal P_A \defi \{J \subset A: a_J > 0\}, \mathcal P \defi = \mathcal P_I .\]
For $n \in \N$, we set $\Sigma_N \defi = \{1, . . . , 2^N\}$. We also fix positive real numbers $\g_i,
i \in I$, satisfying $\sum_{i=1}^n \g_i =1$ and write $\Sigma_N^i \defi  
\Sigma_{\g_i N}$ where, for notational convenience, we assume that
$2^{\g_i N}$ is an integer. For $N \in \N$, we label the spin configurations $\sigma$ as
\[ 
\s = (\s_1, . . . ,\s_n), \s_i \in \Sigma_N^i,
\]
that is, we identify $\Sigma_N$ with $\Sigma_N^1\times \cdots \times \Sigma_N^n$. For $A\subset I =\{1,\dots, n\}$ we write 
\[
\mathcal{P}_{A}\overset{\mathrm{def}}{=}\mathbb{\{}J\subset A:a_{J} >0\mathbb{\}}, \quad \a(A) \defi \sum_{J\in {\cal P}_A} a_J,\quad \g(A) \defi \sum_{i\in A} \g_i, 
\]
and shorten $\mathcal{P}\overset{\mathrm{def}}{=}\mathcal{P}_{I}$. 

For $j = 1,\dots, n$ we set $\Sigma_{N}^{j} = \{1, \dots, 2^{\g_j N}\}$ and identify $\Sigma_{N}$ with \hbox{\small{$\Sigma_{N}^{1}\times\cdots\times\Sigma
_{N}^{n}$}}. For \hbox{\small{$J\subset I$}} with \hbox{\small{$J=\left\{  j_{1},\ldots,j_{k}\right\}$}} and \hbox{\small{$\ j_{1}<j_{2}<\ldots<j_{k}$}} we write
\hbox{\small{$\Sigma_{N,J}\overset{\mathrm{def}}{=}\prod\nolimits_{s=1}^{k}\Sigma_{N}^{j_{s}}$}}. For \hbox{\small{$\tau \in \Sigma_{N, J}$}} and \hbox{\small{$J'\subset J$}} we write 
\hbox{\small{$\tau_{J'}$}} for the projected configuration \hbox{\small{$(\tau
_{j};\; j\in J')$}}.

Our spin glass hamiltonian is defined as
\begin{equation} \label{Hamiltonian}
X_{\sigma}=\sum_{J\in\mathcal{P}}X_{\sigma_{J}}^{J}, 
\end{equation}
where the $X_{\s_J}^J, J\in \mathcal P, \s_J\in \Sigma_{N,j}$ are independent centered gaussian random variables with variance $a_JN$. The $X_\sigma$ are then
gaussian random variables, but they are correlated. $\PP$ and $\E$ will denote respectively probability and expectation  with respect to these random variables.  

The GREM corresponds to the case where subsets in $\mathcal P$ are "nested", i.e. 
\beq \label{gremdef}\mathcal P =\{J_1, \dots, J_m\}, \quad J_m \defi \{1,\dots, n_m\}, \eeq
for an increasing sequence $(J_\cdot)$. In the GREM case the natural metric on $\Sigma_N$ coming from the covariance structure 
\[
d(\s,\s') \defi \sqrt{\E\left[ (X_\s-X_{\s'})^2\right]}
\]
is an {\it ultrametric}, meaning that it satisfies the strenghtened inequality 
\[ 
d(\s,\s') \leq \max_{\s''}\big\{ d(\s,\s''), d(\s', \s'')\big\}.
\]
Remark that such a strenghtening of the triangle inequality is satisfied for distances on hierarchical spaces (e.g. trees), hence the identification of the GREM
with the {\it hierarchical models}. In the general case \eqref{Hamiltonian} considered here, it is easily seen that the natural distance induced by the covariance structure is no longer
an ultrametric. (To visualise things throughout, we suggest the reader to keep in the back of her mind
the paradigmatic nonhierarchical model with $n=3$ and $\mathcal P = \{\{1,2\}, \{1,3\}, \{2,3\}\}$, that is where $X_{\s} = X_{\s_1, \s_2}^{\{1,2\}}+X_{\s_1,
\s_3}^{\{1,3\}}+X_{\s_2, \s_3}^{\{2,3\}}$.)\\

Any of our models can be  "coarse-grained" in many ways into a GREM. For this, consider strictly increasing sequences of subsets of $I: \emptyset = A_o \subset
A_1 \subset \dots \subset A_K = I$. We do not assume that the $A_i$ are in $\mathcal P$. We call such a sequence a {\it chain} $\boldsymbol T = (A_o, A_1, \dots,
A_K)$. We attach weights 
\[
\hat a_{A_j} \defi \a(A_j\setminus A_{j-1}).
\]
Evidently, $\sum_{j=1}^K \hat a_{A_j} = 1$, and if we assign random variables $X_{\s}(\boldsymbol T)$ according to $\eqref{Hamiltonian}$ we arrive after an
irrelevant renumbering of $I$ at a GREM of the form \eqref{gremdef}. In particular, the corresponding metric $d$ is an ultrametric. 

We write $\texttt{tr}(\cdot)$ for averaging over $\Sigma_N$ (i.e. the coin tossing expectation over $\Sigma_N$). For a function $x: \Sigma_N \to \R$, set 
\[ 
Z_N(\be, x)\defi \texttt{tr} \exp[\be x], \qquad f_N(\be, x) \defi {1\over N} \log Z_N(\be, x),
\]
and define the usual finite $N$ partition function and free energy respectively by
\[ 
Z_N(\be) = Z_N(\be, X), \qquad f_N(\be) = f_N(\be, X).
\]
The following is the main results obtained in \cite{bokis} for the limiting free energy of nonhierarchical GREMs: 
\begin{teor}[Bolthausen and Kistler, \cite{bokis}] The limit
\beq f(\be) \defi \lim_N f_N(\be)\eeq
exists, and coincides with $\lim_{N\to \infty}\E f_N(\be)$. Moreover, $f(\be)$ is the free energy of a GREM. More precisely, there exists a chain $\boldsymbol T$ such that 
\beq 
f(\be) = f(\be, \boldsymbol T), \quad \be\geq 0.
\eeq
$f(\be, \boldsymbol T)$ is minimal in the sense that
\beq 
f(\be) = \min_{\boldsymbol S} f(\be, \boldsymbol S),
\eeq
where the minimum is taken over all chains $\boldsymbol S$.
\end{teor} 
According to the above Theorem, the limiting free energy of {\it any} nonhierarchical model always coincides with that a certain hierarchical counterpart. It is
therefore a natural question up to which extent the random systems associated to a nonhierarchical model are genuinely
ultrametric. In this second and concluding work we address exactly this issue. More precisely, we provide a complete description of the Gibbs measure associated
to a hamiltonian \eqref{Hamiltonian}, which is the random probability on $\Sigma_N$ given by 
${\mathcal G}_{\be, N}(\s) \defi Z_N^{-1}(\be) \exp\big[\be X_\s\big]$. We prove here that the configuration space $\Sigma_N$ is
hierarchically organized under $\PP\otimes \mathcal G_{\be, N, \cdot}$, provided the hamiltonian satisfies some additional assumptions of {\it irreducibility}, while this
is not true in the most general case (a precise statement of the irreducibility 
condition will be given below). More precisely, if we write $\left<\cdot\right>_{\be, N}^{\otimes 3}$ for average with respect to the quenched Gibbs measure over the replicated space
$\Sigma_N^3$, we have

\begin{teor}[Ultrametricity.] \label{ultrametricity_theorem} If the hamiltonian is irreducible,  
\[ 
\lim_{N\to \infty} \E \left< d(\s,\s') \leq \max_{\s''}\Big\{ d(\s,\s''), d(\s', \s'')\Big\}\right>_{\be, N}^{\otimes 3} = 1, 
\]
for $\be$ large enough.
\end{teor} 

The strategy to prove Theorem \ref{ultrametricity_theorem} relies on the observation that already the set of relevant configurations, those $\s$'s with energies "close" (we will make this precise)
to the ground state, is hierarchically organized in the large $N$-limit. Given the absence of chaotic behavior in the temperature, a feature which turns out to be shared 
by any of the models of Derrida's type, nonhierarchical GREMs included, this approach is particularly efficient, and additionally clarifies the coarsening of the
hierarchical structure depicted in \cite{bk2} for the GREM. This self organization is 
outcome of an energy/entropy competition, which, provided the irreducibility of the hamiltonian, leads to a "suppression and propagation of structures", as we
shall elucidate.  Some other notation: we set the overlap $q(\s,\s')$ of two configurations $\s,\s' \in \Sigma_N$ to be the subset of $I$ where they  agree, $q(\s,\s') \defi \{i\in I:
\s_i =\s_i'\}$.\\

\subsection{Suppression} We consider some models whose limiting free energy coincides with that of a Random Energy Models (the REM) which however display different microscopic behavior at the level of the Gibbs measure.
\begin{itemize}
\item[M1.] The first model is a hierarchical GREM with two levels, i.e. \hbox{\small{${\cal P} = \big\{ \{1\}, \{1,2\}\big\}$}} and parameters such that the
optimal chain is  ${\bf T} = \left\{ \{1,2\} \right\}$. In this case, some easy evaluations of gaussian integrals yield  
\[ \bea \label{gaussian_integral}
\lim_{N\to \infty} \PP\left[\exists \;\text{relevant}\, \s, \tau \in \Sigma_{N},  q(\s, \tau) = \{1\} \right] = 0 
\eea \]
(this also holds if we require $q(\s,\tau) =\{2\}$) implying that the relevant configurations either differ on both spins, in which case the random variables $X_\cdot$ are independent, or they coincide. This explains the REM-like
behavior also on the finer scale of the Gibbs measure. This observation is in fact the crux of our approach in the more general case of nonhierarchical models,
as the following model indicates. 
\item[M2.] Consider for example the case where \hbox{\small{${\cal P} = \big\{ \{1\},\{2\},\{1,2\}\big\}$}}
with ${\bf T} = \left\{\{1 ,2\}\right\}$. Also here, in the large $N$ limit, given two relevant configurations $\s,\tau\in \Sigma_N,\; \s_1 = \tau_1$ implies $\s_2 =
\tau_2$ (and the other way around) on a set of $\PP$-probability close to unity: this kind of  (nonhierarchical) dependencies  is also {\it suppressed}, and the
overlap of relevant configurations is either the  full or the empty set. That this is not
always the case may be seen by inspection of the following nonhierarchical model. 
\item[M3.] Consider \hbox{\small{${\cal P} = \big\{ \{1\},\{2\}\big\}$}} with ${\bf T} =\{\{1,2\}\}$: with non vanishing
probability, one can find relevant $\s,\tau, \tau' \in \Sigma_N$  such that $q(\s,\tau) = \{1\}$ and 
$q(\s,\tau') = \{2\}$; this kind of nonhierarchical dependencies are {\it not} suppressed. A moment thought shows that is due to the fact that $\mathcal P$
consists of two disjoint sets, $\{1\}$ and $\{2\}$: this does not prevent the system to display 'clustering' at the level of the free energy, but it
does have an impact on the behavior of the Gibbs measure (which, being a product measure on $\Sigma_{N,1}\times \Sigma_{N,2}$, must obviously contradict the ultrametricity). \\
\end{itemize}

\subsection{Propagation} 
\begin{itemize}
\item[M4.] Consider again a two-levels GREM, but with underlying parameters such that \hbox{\small{${\bf T} = \big\{\{1\}, \{1,2\}\big\}$}}. It is then easy to see that the probability that there exist relevant configurations $\s,\tau \in \Sigma_N$
such that $q(\s,\tau) =2$ is vanishing, but not if we require $q(\s, \tau) =1$: given that $\s,\tau \in \Sigma_N$ coincide on the second index ($\s_2 = \tau_2$)  then automatically on the first as well, in which case the two configurations coincide.
\item[M5.] Finally, let ${\cal P} = \big\{ \{1\},\{2\},\{2,3\}\big\}$ and ${\bf T} = \big\{ \{1\}, \{1,2,3\}\big\}$. In this case, also on the finer
level there is {\it clustering} on the second level (e.g. $\s_2 = \tau_2$ implies $\s_3 = \tau_3$), but it is not true that $\s_2 =
\tau_2$ implies $\s_1 = \tau_1$ nor $\s_3 =\tau_3$ implies $\s_1=\tau_1$. Intuitively, the lack of a "linking bond" from the second branch to the first prevents 
the coincidence of the spins indexed by $A_2$ to propagate "upwards" to the spins indexed by $A_1$. \\
\end{itemize}

The proof of Theorem \ref{ultrametricity_theorem} boils down to making the above explicit and rigorous in the general case. In fact, we will prove a stronger result, Theorem \ref{main_theorem} below, which confirms the "full Parisi Picture" for nonhierarchical,
irreducible models (and not only the ultrametricity): {\it i.} the law of the limiting Gibbs measure is given by the Poisson-Dirichlet distribution. {\it ii.} The law of the
overlaps is given by the coalescent introduced in \cite{boszni_ruelle}. {\it iii.} Overlaps and Gibbs measure are independent.\\ 

In order to formulate precisely the Main Theorem, we need an infrastructure which allows to attach marks independently to a Point Process: the way we do this is explained in great generality in
Section \ref{infrastructure} (and might be of independent interest), and specified to the setting of nonhierarchical models in Section \ref{ggrem_infrastructure},
where the irreducibility conditions and the Main Theorem are also
stated. The crucial steps behind the Main Theorem are highlighted in Section \ref{main_steps}, while the proofs are collected in Sections \ref{energy_levels}-\ref{proof_main_theorem}. \\

\section{The Parisi Picture for nonhierarchical GREM} \label{main_theorem_section}
\subsection{Attaching independent marks to a Point Process}\label{infrastructure}
Let $X$ be a locally compact space with countable base (lccb for short). We write ${\mathcal M}(X)$ for the set of Radon measures, and ${\mathcal M}_p(X)$ for the subset
of pure point measures. We also write $X^{(2)}$ for the set of two-element subsets of $X$. Clearly, $X^{(2)}$ is a lccb, too [{\it we can identify it for instance with
$(X^{2}\setminus D)/\sim$ , where $D$ is the diagonal $\{(x,x): x\in X\}$ and $(x,y) \sim (y,x)$}]. We write $\pi$ for the projection $(X^{2}\setminus D) \to X^{(2)}$.
\\

Any Radon measure $\mu$ on $X$ induces a Radon measure $\mu^{(2)}$ on $X^{(2)}$ by first taking the product measure $\mu\times \mu$ on $X^{2}$, restrict it to the
complement of the diagonal, and project it on $X^{(2)}$.  We write $\psi: {\mathcal M}(X) \to {\mathcal M}(X^{(2)})$ for this mapping. The image of a pure point measure is
clearly a pure point measure. Also, if $K$ is a compact subset of $X$, then $\rho_K: {\mathcal M}(X) \to {\mathcal M}(K)$ is given by restricting $\mu \in {\mathcal
M}(X)$ to $K$. This transforms pure point measures to pure point measures, of course. For compact $K$ and $\mu \in {\mathcal M}(K)$, the total mass $|\mu|$ of $\mu$ is
finite. If $\mu \in {\mathcal M}_p(K)$, this is just the number of points of the point measure $\mu$. {\it [It is easy to see that $\psi$ is continuous in the vague topology.
For this, consider  a continuous function with bounded support $f: X^{(2)} \to \R$. Then $f\circ \pi$ has compact support on $(X^{2}\setminus D)$ and therefore, we can extend it (by $0$) to a
function of compact support on $X^2$, which we still write as $f\circ \pi$. Assume $\mu_n \to \mu$ vaguely, for $\mu_n, \mu \in {\mathcal M}(X)$. Then
\hbox{\small{$\lim_{n\to \infty} \int f d\mu_N^{(2)} = \lim_{n\to \infty} \int f\circ \pi d\mu_n = \int f d\mu^{(2)}.$]}}}\\

\noi Let now $F$ be a finite set. If $Y$ is a lccb, we define ${\mathcal M}_{mp}(Y\times F)$ to be the subset of ${\mathcal M}_p(Y\times F)$
consisting of measures with the property that its marginal on $Y$ is in ${\mathcal M}_p(Y)$. In other words, the measures in ${\mathcal M}_{mp}(Y\times F)$ are of the
form 
\[ \sum_i \de_{\{y_i, a_i\}}, \quad y_i \in Y, \; a_i \in F\]
where the $y_i$ are all distinct, and $\{y_i\}$ is locally finite. It is clear that ${\mathcal M}_{mp}(Y\times F)$ is a measurable subset of ${\mathcal M}(Y\times F)$.
Weak convergence of probability measures on ${\mathcal M}_{mp}(Y\times F)$ refers to weak convergence of their extensions to ${\mathcal M}(Y\times F)$. \\
If $K\subset \subset Y$ is a compact subset, then we set $\hat \rho_K: {\mathcal M}_{mp}(Y\times F) \to {\mathcal M}_{mp}(K\times F)$ by taking the restriction. It is clear that
any probability measure $P$ on ${\mathcal M}_{mp}(Y\times F)$ is uniquely determined by the family $P \hat \pi_K^{-1}$, $K$ compact in $Y$. Furthermore, for any
consistent family of such probability measures $P_K$ on ${\mathcal M}_{mp}(K\times F)$, $K \subset\subset Y$, there is a unique probability measure $P$ on ${\mathcal
M}_{mp}(Y\times F)$ with $P \hat \pi_K^{-1} = P_K$. Consistency means that for $K\subset K'$ one has $P_K' \hat\pi_{K', K}^{-1} = P_K$, where $\pi_{K', K}: {\mathcal
M}_{mp}(K'\times F) \to {\mathcal M}_{mp}(K\times F)$. This follows easily from Kolmogoroff's Theorem. It suffices to have the $P_K$ consistently defined for a sequence
of compacta $(K_n)$  with $K_n \uparrow Y$. \\

Let $\N^{(2)} \defi \{(i,j): \; i,j \in \N, i< i\}$. We consider probability measures $Q$ on $F^{\N^{(2)}}$ which have the property that they are invariant under finite
permutations: a permutation ${\mathfrak s}: \N \to \N$ which leaves all the number except finitely many fixed induces a mapping $\phi_{\mathfrak s}: F^{\N^{(2)}} \to
F^{\N^{(2)}}$ in a natural way. We call $Q$ invariant if it is invariant under all such $\phi_{\mathfrak s}$. \\

Given a sequence of distinct points $\bold x = (x_1, \dots, x_N)$ in some compact set $K$, and $\bold f = (f_{ij}, 1\leq i < j \leq N),\; f_{ij} \in F$, we put 
\[L(\bold x, \bold f) \defi \sum_{i< j} \de_{x_i, x_j, f_{ij}} \in {\mathcal M}_{mp}(K^{(2)}\times F).\]
For fixed $\bold x$, this defines a mapping $L(\bold x, \cdot):F^{\hat N} \to {\mathcal M}_{mp}(K^{(2)}\times F)$, where $\hat N\defi \{(i,j): 1\leq i< j \leq N\}$. Given an invariant $Q$ on $F^{\N^{(2)}}$, $N \in \N$, we write
$Q_N$ for its restriction on $F^{\hat N}$. The $Q_N L(\bold x, \cdot)^{-1}$ is a probability measure on ${\mathcal M}_{mp}(K^{(2)}\times F)$, depending still on $N$ and
$\bold x$. We denote it by $\Pi(N, \bold x; \cdot)$. By the invariance property of $Q$, it only depends on the set $\{x_1, \dots, x_N\}$ (or on $\sum \de_{x_i}$).
Therefore, for fixed $N,\; \Pi(N, \cdot; \cdot)$ is a Markov Kernel from ${\mathcal M}_{p, N}(K) \defi \{\mu \in {\mathcal M}_{p}(K): \; |\mu| = N\}$ to ${\mathcal
M}_{mp}(K^{(2)} \times F)$. \\

With $X$ lccb, and $P$ a probability on ${\mathcal M}_p(X)$, we choose compacts $(K_n)$ with $K_n \uparrow X$. We also write $P_n \defi P \rho_{K_n}^{-1}$
on ${\mathcal M}_p(K_n)$. Then we define $\hat P_n$ on ${\mathcal M}_{mp}(K_n^{(2)} \times F)$ by 
\[ 
\hat P_n \defi \int P_n(d\mu) \Pi(|\mu|, \mu; \cdot).
\]
This satisfies the above required consistency property, and therefore gives rise to a probability measure on ${\mathcal M}_{mp}(X^{(2)} \times F)$, which evidently does
not depend on the sequence $(K_n)$ chosen, and is denoted by $P \sqcap Q$. \\

\subsection{Nonhierarchical GREM and Main Theorem}\label{ggrem_infrastructure}
We now put the nonhierarchical models into the above setting. \\

First, we specify $F$ further by choosing it to be the set $2^I$ of subsets of $I =\{1,\dots, n\}$. Also, we recall from \cite{bokis} that the free energy of a nonhierarchical GREM is determined by a chain ${\bf T}= (A_0, A_1, \dots, A_K), A_0 = \emptyset
\subset A_1 \subset \dots \subset A_K = I$. The chain is essential to construct the sequence of inverse of temperatures ${\boldsymbol \be} = (\be_0, \be_1,\dots, \be_K)$, $\be_0 = 0 <
\be_1 < \dots < \be_K < \be_{K+1} = \infty$ at which the free energy undergoes a phase transition. For $m=1,\dots, K-1$, we denote by ${\bf T}^{(m)} = (A_0, \dots, A_{m-1}, A_m)$ the chain restricted to
the first $m$-levels.  A fixed realization of the Hamiltonian induces an element of
${\mathcal M}_{mp}\left((\R^+)^{(2)} \times 2^I\right)$ by setting
\[ 
\sum_{\s,\s'} \de_{\{\mathcal G_{N, \be}(\s), {\mathcal G}_{N, \be}(\s'); q(\s,\s')\}}.
\]
We denote by $\Xi_{N,\be}$ its distribution under $\PP$. Analogously, by $\Xi_{N,\be}^{(m)}$ we understand the law of the element of ${\mathcal M}_{mp}\left((\R^+)^{(2)} \times
2^{A_m}\right)$ induced by the {\it $m^{th}$-marginal} of the Gibbs measure, the latter being the collection of points
\[ 
{\mathcal G}_{\be, N}^{(m)}(\tau) \defi \sum_{\s\in \Sigma_N: \s_{A_m} = \tau} {\mathcal G}_{\be, N}(\s), \quad \tau \in \Sigma_{N, A_m}. 
\]
Our main result is to determine the weak limits of the measure $\Xi_{\be, N}$ (and $\Xi_{\be, N}^{(m)}$) describing at the same time the limiting
Gibbs distribution, and the limiting overlap structure, where the latter will be given in terms of the coalescent on $\N$
introduced in \cite{boszni_ruelle}. This is a continuous time Markov
process $(\psi_t, t\geq 0)$ taking values in the compact set of partitions on $\N$. We call a partition $\mathcal C$ finer than $\mathcal D$, in notation ${\mathcal C} \succ \mathcal
D$, provided that the sets of $\mathcal D$ are unions of the sets of $\mathcal C$. The process $(\psi_t, t\geq 0)$ has the following properties: {\it i.} If $t\geq s$  then $\psi_s \succ \psi_t$. 
{\it ii.} The law of $(\psi_t, t\geq 0)$ is invariant under permutations. {\it iii.} $\psi_0 = 2^\N$. We denote the equivalence relation associated with $\psi_t$ by $\sim_t$. Given this coalescent, a sequence ${\bold t} = (t_0, \dots, t_K)$ of {\it times} $t_0 = 0 < t_1 < t_2
\dots < t_{K-1} < t_K = \infty$, and a chain $\bf T$ as above, we attach to each pair $i< j$ of natural numbers randomly the $A_{K-k}, 1\leq k \leq K$ (and
only these) where $k\defi \min\{l: i \sim_{t_l} j\}$. This defines a law $Q_{{\bf T}, {\bf t}}$ on $(2^I)^{\N^{(2)}}$. The law $Q_{{\bf T^{(m)}},
{\boldsymbol{t^{(m)}}}}$ is constructed analogously, outgoing from the sequence of times $\boldsymbol{t^{(m)}} = \{t_1, \dots, t_m\}$ and marks $A_{m-k}, 1\leq k\leq m$. 
 
\begin{itemize}
\item {\bf Condition $\bf c$}. For every $j=1,\dots, K$ and $A \subsetneq A_{j}\setminus A_{j-1},\; \exists J\in {\cal
P}_{A_j}\setminus {\cal P}_{A\cup A_{j-1}}, J' \in {\cal P}_{A\cup A_{j-1}}\setminus {\cal P}_{A_{j-1}}$ such that $(J\cap J') \setminus
A_{j-1} \neq \emptyset$,
\item {\bf Condition $\bf c'$}.  For all $j=2,\dots, K$ there exists $s\in A_{j-1}\setminus A_{j-2}, \; J\in {\cal P}_{A_j}\setminus {\cal
P}_{A_{j-1}}$ such that $J\ni s$.
\end{itemize}

These are the {\it irreducibility conditions}. In some loose sense, they ensure that the underlying graph is "connected enough". 
(To shed some light on this presumably opaque conditions, consider the models from the introduction: it is not difficult to check that the models M1, M2 and M4 satisfy both conditions $\bf c$ and $\bf c'$;
on the other hand, the model M3 does not satisfy condition $\boldsymbol c$, while the model M5 does not satisfy condition $\boldsymbol c'$. Therefore, none of our
results apply for the models M3 and M5, but for M1, M2 and M4.) Henceforth, we will assume that the hamiltonian is {\it irreducible}, meaning that it satisfies both $\bf c$ and $\bf c'$.\\

For a Poisson Point Process $(\eta_i, i\in \N)$ of density $x t^{-x-1}dt$ on $\R^+$ with $x \in (0,1)$, we understand by $(\overline \eta_i, i \in \N)$ the normalized
process where $\overline \eta_i = \eta_i / \sum_j \eta_j$, and denote by $P_x$ its law. \\

Given a hamiltonian  with chain ${\bf T} = \{A_0, A_1, \dots, A_K\}$ and associated sequence of phase transitions
${\boldsymbol \be} = \{\be_0, \be_1, \dots, \be_K\}$ we define the "times" through $t_j = \log(x_K/x_{K-j}),\;
x_j = x_j(\be) = \be_j/\be$.  The following is our Main Theorem:\\

\begin{teor}[Parisi Picture] \label{main_theorem} Assume the hamiltonian is irreducible. Then,   
\begin{itemize} 
\item if $\be >\be_K, \;\lim_{N\to \infty} \Xi_{N, \be} = P_{x_K} \sqcap Q_{{\bf T}, {\bf t}}$ weakly. 
\item if $\be>\be_m,\; \lim_{N\to \infty}\, \Xi_{N, \be}^{(m)} = P_{x_m} \sqcap Q_{{\bf T}^{(m)}, {\bf t}^{(m)}}$ weakly. 
\end{itemize}
\end{teor}

According to Theorem \ref{main_theorem}, the only possible "marks" in the large $N$-limit are thus the ones from the chain ${\bf T}$: this is
a stronger version of the ultrametricity, and in fact, one can easily see that it automatically entails Theorem \ref{ultrametricity_theorem}.

\subsection{Outline of the proof of the Main Theorem} \label{main_steps}
We first introduce some notations. \\

\noi{\bf Generalities.} We will refer to $(a_J, \g_i; J\in {\cal P}, i\in I)$ as "underlying parameters". \\

For $j=1\dots K$ we write \[\D_j \defi \a(A_{j})-\a(A_{j-1}), \quad G_j \defi
\g(A_j)-\g(A_{j-1}).\]
 
For a subset $A\subsetneq A_{j}\setminus A_{j-1}:$
\[\bea 
&\hspace{1cm} \widehat {\cal P}_{A,j} \defi {\cal P}_{A\cup A_{j-1}}\setminus {\cal
P}_{A_{j-1}},\qquad \widehat {\cal P}_{A,j}^c \defi {\cal P}_{A_j}\setminus \widehat {\cal P}_{A\cup A_{j-1}},\\
& \widehat \a_j(A)\defi \a(A\cup A_{j-1}) - \a(A_{j-1}),\qquad \widehat \a_j^c(A) \defi \Delta_j -\widehat \a_j(A), \\
&a_{N,j}(A) \defi \be_j \widehat \a_j(A) N- {1\over 2 \be_j}\log N +{1\over \be_j}\log \be_j \sqrt{2\pi \widehat \a_j(A)},
\eea \]
and  $a_{N,j} \defi a_{N,j}(A_{j}\setminus A_{j-1})$.\\

Finally, for $m=1,\dots, K$ we set 
\beq \label{aenne}
a_N^{m}\defi \sum_{j=1}^m a_{N,j} + \sum_{j=m+1}^K \left[{\be\over 2} \Delta_j N + {G_j\over \be} N \log 2\right], \qquad 
a_N \defi a_N^K.
\eeq

\noi{\bf Random variables.} By $(Y_J, J\in {\cal P})$ we denote a family of independent centered gaussians, $\E\big( Y_J^2 \big) = a_J$, 
and shorten notations by setting
\[ \bea
&Y_j \defi \sum_{J \in {\cal P}_{A_j}\setminus {\cal P}_{A_{j-1}}} Y_J,\quad \overline Y_j \defi \sqrt{N} Y_j -a_{N,j}, \quad \widehat Y_{j}
\defi \sum_{l=1,\dots, j} \overline Y_l,\\
& \hspace{2cm} Y_{j, A} \defi \sum_{J \in \widehat {\cal P}_{A,j}} Y_J, \qquad Y_{j, A}^c \defi  \sum_{J \in \widehat {\cal P}_{A,j}^c} Y_J.
\eea \]
By $(Z_J)$ we denote a faimly of random variables, independent of the $(Y_J)$ but with same distribution. We write
analogously $Z_{j,A}, Z_{j, A}^c, \overline Z_j, \widehat Z_j$. \\

\noi For $\s\in \Sigma_{N,A_j}$ we write \hbox{\small{$\s =
(\s(1),\dots, \s(j))$}} with $\s(k) = (\s_i; i\in A_{k}\setminus A_{k-1})$ and 
\[ \bea 
&X_{\s} = \sum_{j=1}^K X_{\s(1),\dots, \s(j)}, \quad X_{\s(1),\dots, \s(j)} \stackrel{\text{def}}{=}
\sum_{J\in {\cal P}_{A_j}\setminus {\cal P}_{A_{j-1}}} X_{\s_J}^J\\
& \overline X_{\s(1),\dots, \s(j)} \defi X_{\s(1),\dots, \s(j)} - a_{N,j}, \quad \widehat X_{\s(1),\dots, \s(j)} \defi \sum_{l=1}^j \overline
X_{\s(1),\dots, \s(l)}. \eea\]

\noi{\bf Critical subsets.} For $B\subset A$ let  
\[
\rho(B, A) \defi \sqrt{2\log 2 {\g(A)- \g(B) \over \a(A)- \a(B)}},\quad \hat \rho(B, A) \defi \min_{A: A\supset B, A\neq B}
\rho(A,B).
\]
The sequences $(A_1, \dots, A_K)$ and $(\be_1,\dots, \be_K)$ are constructed by recursion (cfr. \cite{bokis}). They enjoy the following properties: first, $\be_j =
\hat \rho(A_j)$; second, for all $A \supset A_{j-1}$ with $\be_j = \rho(A_{j-1}, A)$ one has $A\subset A_j$, i.e. $A_j$ is maximal
with $\be_j = \rho(A_{j-1}, A_j)$. Accordingly, there may be strict $A\subsetneq A_{j}\setminus A_{j-1}$ such that 
\[
\rho(A_{j-1}, A\cup A_{j-1}) = \be_j \; \left(\text{i.e.} \quad {\g(A_j) - \g(A_{j-1} \cup A) \over  \widehat \a_j(A)} = {\be_j^2 \over 2}\log 2\right),
\] in which case we call the subsets {\it critical}.\\

\noi{\bf Ultrametricity.} We say that $\s, \tau \in \Sigma_{N,A_j}$ (for some $j = 1, \dots, k$) form a {\it non ultrametric couple} if there exists $k=1,\dots, j$ and $s\in A_{k}\setminus A_{k-1}$ such that 
$\s_s = \tau_s$ but $\s_{A_k} \neq \tau_{A_k}$ (i.e. $\s_r \neq \tau_r$ for some $r\in A_k$).\\

\noi{\bf Point processes.} PP will stand for {\it Point Process} and PPP
for {\it Poisson Point Process}. For a PP $(y_i,\; i\in \N)$ such that $\sum_i y_i < \infty$ almost surely, we may consider new points given by  $\overline y_i \defi {y_i/\sum_j y_j}$, and write
${\cal N}\big((y_i, i\in \N)\big) \defi (\overline y_i,\; i\in \N)$ for the normalization procedure. We also encounter superpositions of PP in which case it
is notationally useful to introduce multi-indices ${\bf i} \defi (i_1, \dots, i_j)$ (for $j\in \N$ to be specified) and denote by ${\bf
i}_k = (i_1, \dots, i_k)$ the restriction to the first $k$ indeces, $k<j$.\\

\noi{\bf Constants.} We denote by $const$ a strictly positive constant, not necessarily the same at different occurences. For $X,Y >0$ we write $X \lesssim Y$ if $X\leq const \times Y$  (for sequences: $X_N\lesssim Y_N$ stands for $X_N \leq
const \times Y_N$ for $N\geq N_o$ for some $N_o\in \N$). \\

The first step in the proof of the Main Theorem will be to control the {\it energy levels}: consider for $j=1,\dots, K$ the collection
\hbox{\small{$\big(\widehat X_{\s(1),\dots, \s(j)}; \s\in \Sigma_{N,A_j}\big)$}} - the process of the energy levels
corresponds to the choice $j=K$.    
\begin{prop} \label{nonultra_prop} Let $\,\Diamond \subset \R$ be a compact set. To given $\vare >0$, 
$$\PP\left[ \exists \; \text{nonultrametric couples}\; \s,\tau\in \Sigma_{N, A_j}:\; \widehat X_{\s(1),\dots, \s(j)}, \widehat X_{\tau(1),\dots, \tau(j)} \, \in \Diamond \right] \leq \vare, $$
for large enough $N$.
\end{prop}

The configurations which survive the passage to the limit - in this sense: relevant - must therefore satisfy hierarchical constraints; in fact, the Proposition implies that the overlap of configurations
falling into given compacts are, with probability arbitrarily close to unity, in the chain, and in the chain only (or, more precisely, in the chain restricted to the first $j$ sets, when considering the "partial energies"). It is thus very natural to expect that their statistics are given in the thermodynamical limit by the hierarchical models. To formalize this, we first observe that collections of points such as the \hbox{\small{$\big(\widehat X_{\s(1),\dots, \s(j)}; \s\in \Sigma_{N,A_j}\big)$}} naturally induce elements of 
\hbox{\small{${\cal M}_{mp}\big(\R^{(2)} \times 2^{A_j}\big)$}}, namely 
\[ 
{\cal N}_{N,j} \defi \sum_{\s, \tau \in \Sigma_{N,A_j}} \de_{\{\widehat X_{\s(1),\dots, \s(j)},\widehat X_{\tau(1),\dots, \tau(j)}; q(\s,\tau)\}}. 
\] 
We denote by $\widehat X_{N, j}$ the law of such an element. \\

The "limiting object" will be given in terms of the Derrida-Ruelle processes \cite{ruelle}. Consider a PP $(y_{\bf i}, {\bf i \in\N}^j)$ with the following properties: {\it i.} For $l=1,\dots, j$ and multi-index ${\bf i}_{l-1}$, the point process \hbox{\small{$(y_{{\bf i}_{l-1},
i_l}^l;\; i_l\in \N)$}} is poissonian with density \hbox{\small{${\mathcal C}_{l} \cdot \be_l e^{-\be_l t} dt$}} on $\R$. {\it ii.} The $y^l$ are independent for different $l$.
{\it iii.} $(y_{{\bf i}_{l-1}, i_l}^l;\; i_l \in \N)$ are independent for different ${\bf i}_{l-1}$. {\it iv.} If \hbox{\small{$A_{l}\setminus A_{l-1}$}} contains no critical subsets, then ${\mathcal C}_{l} = 1$, otherwise   
\[ 
{\mathcal C}_{l} = \PP\left[\left\{{Y_{l, A}\over \widehat \a_l(A)} - {Y_{l, A}^c\over \widehat \a^c_l(A) }\leq 0 \right\} \; \forall A \subsetneq A_{l}
\setminus A_{l-1}, \; A\;\text{is critical} \right]. 
\]
Given two points $y_{\bf i}$ and $y_{{\bf i}'}$, we define their overlap $q_{\bf i, \bf i'}$ to be $A_{m}$ where $m =
\max\big\{l\leq j:\; {\bf i}_l = {\bf i}'_l\big\}$. A fixed realization of the PP induces naturally an element ${\cal N}_{j} \in {\cal M}_{mp}\big(
\R^{(2)} \times 2^{A_j}\big)$ whose law is denoted $\widehat X_j$. 

\begin{prop} \label{partial_energies}  $\widehat X_{j,N}$ converges weakly to $\widehat X_j$. \end{prop}

It easily follows from the above Theorem (with $j=K$) that the process of extremes associated to the energy levels of an irreducible hamiltonian coincides, in the thermodynamical limit, with that of a
hierarchical model. In fact, denoting by $\widetilde X_K$ the first marginal of $\widehat X_K$ (that is: the law of the point process $\sum_{\boldsymbol i \in
\N^K} \de_{y_{\boldsymbol i}}$) we have:

\begin{cor}\label{extremal_process_theorem}
Consider an irreducible hamiltonian ${X_{\sigma}, \sigma \in \Sigma_N}$ and let $a_N$ be given by \eqref{aenne}. Then, with the above notations: the extremal process
\[
\sum_{\sigma\in \Sigma_N} \de_{X_\sigma - a_N} 
\] 
converges weakly to $\widetilde X_K$.
\end{cor}

\begin{rem}
The constants $\mathcal C_\cdot$ which appear in Property {\it iv)} encode a subtle optimal strategy for the energy/entropy competition in the presence of critical subsets, which loosely goes as follows:
it turns out that a configuration $\s\in \Sigma_{N, A_j}$ is relevant as long as \hbox{\small{$\sum_{k=1}^j X_{\s(1),\dots, \s(k)}\approx \sum_{k=1}^j a_{N,k}$}} (in sub-logarithmic order). Typically, this feat is achieved by simply having all the partial components of the sum to be at their optimal value, 
$X_{\s(1),\dots, \s(k)} \approx a_{N,k}$. It however turns out that in the presence of a critical subset $A$ at the level $k$, say, this is not enough: the optimal strategy has to be 
refined by lowering the sub-energies at the level of the critical subsets, \hbox{\small{$\sum_{J \in \mathcal P_{A\cup A_{k-1}}\setminus \mathcal P_{A_{k-1}}} X_{\s_J}^J \approx a_{N,k} - O(\sqrt{N})$}}, and have the
complement to make up for the energy loss, i.e. \hbox{\small{$\sum_{J \in \mathcal P_{A_{k-1}}\setminus  (\mathcal P_{A\cup A_{k-1}}\setminus \mathcal P_{A_{k-1}})} X_{\s_J}^J  \approx a_{N,k} + O(\sqrt{N})$}}. In other words, one additionally has to require (by a truncation procedure) that
\[
\sum_{J \in \mathcal P_{A\cup A_{k-1}}\setminus \mathcal P_{A_{k-1}}} X_{\s_J}^J - \sum_{J \in \mathcal P_{A_{k-1}}\setminus  (\mathcal P_{A\cup A_{k-1}}\setminus \mathcal P_{A_{k-1}})} X_{\s_J}^J  = -O(\sqrt{N}). 
\]
(In the presence of multiple criticalities, the above must then be required for each one of the critical subsets.) It is also interesting to observe that these constants, in a sense the only witnesses of the original "graph structure", do not enter into the law of the Gibbs measure, as they drop out after the normalization. 
\end{rem}
\begin{rem}
There is also an interesting interpretation of the critical constants $\mathcal C_{\cdot}$ in case of a GREM. To see this, consider on an additional probability
space $(\tilde \Omega, \tilde{ \cal F}, \tilde \PP)$ a Brownian Bridge $(\mathcal B(t), 0\leq t \leq 1)$, starting and ending in $0$. The a priori hierarchical structure of the
GREM is reflected in the nestling of the critical subset, $A^{crit}_1 \subsetneq A^{crit}_{2}, \dots, A_{j}^{crit} \subsetneq
A_{l} \setminus A_{l-1}$. Defining the "times" $s_r = \widehat \a_l(A^{crit}_{r})$, for $r = 1, \dots, j$  one can show that the critical
constants are given by \hbox{\small{${\mathcal C}_l = \tilde \PP\Big[ {\mathcal B}(s_1) \leq 0, \dots, {\mathcal B}(s_j) \leq 0\Big]$}}.
This is by no means fortuitous; there is in fact a strong link between the issues addressed in this work and those related to precise second-order
corrections of the maximal displacement of branching brownian motion \cite{bramson}. Contrary to the GREM, there is no "Brownian bridge representation" of the critical constants for genuinely non hierarchical hamiltonians.  
\end{rem}
Coming back to the Gibbs measure, we observe that its distribution is invariant under 'shifts by constants' of the energies; for instance, in the case $\be>\be_K$ we will think
of the Gibbs measure as
\[ 
\mathcal G_{\be, N}(\s) = {\exp[\be X_\s]\over Z_N(\be)} = {\exp[\be (X_\s - a_N) ]\over \sum_{\tau \in \Sigma_N} \exp[\be (X_\tau - a_N) ]} = {\exp[\be \widehat X_{\s(1), \dots, \s(K)} ]\over 
\sum_{\tau \in \Sigma_N} \exp[\be \widehat X_{\tau(1), \dots, \tau(K)}]}
\]
with $a_N \defi \sum_{j\leq K} a_{N,j}$. Under the light of this representation, together with Proposition \ref{partial_energies} (with $j=K$), it should be clear that an important step in the proof 
of the Main Theorem (part $\it a$) will be to check that the normalization  procedure commutes with the $N\to \infty$ limit. (Whereas the claim {\it b} of the Main Theorem will require some analogous reformulation of the marginal of the Gibbs measure).

\section{The energy levels} \label{energy_levels}
\subsection{Localization of the energy levels} \label{proof_existence}
The following estimates are evident: 
\beq \label{evident}
\frac{a_{N,j}}{\D_j N} = \be_j + O(N^{-1}\log N), \quad \exp\Big[- \frac{a_{N,j}^2}{2\D_j N} \Big]
= 2^{- G_j N} \be_j \sqrt{2\pi \D_j N} \big[1+o(1)\big].
\eeq
The next Lemma relates to exponentials of gaussian random variables. Let \hbox{\small{$B>
\be_j$}} and \hbox{\small{$\,B_N \defi B+\epsilon_N$}}, for some $\epsilon_N \to 0$. 
\begin{lem} \label{control}
For any sequence of reals $\phi_1, \dots, \phi_j$ there exists "const" depending on the underlying parameters 
only (not yet on $\phi's$) such that for $N$ large enough
\beq \bea 
&\E\Bigg[\exp\Big( B_N\widehat Y_j\Big);\; \widehat Y_1 \leq \phi_1, \widehat Y_2 \leq \phi_2, \dots,
\widehat Y_j  \leq \phi_j \Bigg] \\
&\hspace{6cm} \lesssim 2^{-\g(A_j) N}\exp\Bigg\{\sum_{l=1}^{j-1}(\be_{l+1} -\be_l)\phi_l
+(B-\be_j) \phi_j\Bigg\}. 
\eea \eeq
\end{lem}
\begin{proof} 
Let $\E_{\overline Y_j}$ stand for expectation w.r.t. $\overline Y_j$. Then
\small{
\beq \bea \label{control_one}
&\E\Bigg[\exp\Big( B_N\widehat Y_j\Big);\; \widehat Y_1 \leq \phi_1, \widehat Y_2 \leq \phi_2, \dots,
\widehat Y_j  \leq \phi_j \Bigg] =\\ 
& = \E\Bigg[ \exp\Big(B_N \widehat Y_{j-1} \Big)\E_{\overline Y_j} \Big[ \exp\big(B_N \overline Y_j \big);\; \widehat
Y_{j-1} + \overline Y_j \leq \phi_j \Big] ;\; \widehat Y_1\leq \phi_1,\dots, \widehat Y_{j-1} \leq
\phi_{j-1}\Bigg]. 
\eea \eeq }
But
\small{
\beq \bea  \label{control_two}
&\E_{\overline Y_{j}}\Big[ \exp\big(B_N \overline Y_j \big);\; \widehat
Y_{j-1} + \overline Y_j \leq \phi_j\Big] = \int_{-\infty}^{\phi_j - \widehat Y_{j-1}} \exp\Bigg[ B_N x 
- {\big(x+a_{N,j}\big)^2\over {2\D_jN}} \Bigg] {dx\over \sqrt{2\pi \Delta_j N}} \\
&\qquad \leq \exp\Bigg[-{a_N,j^2\over 2\Delta_j N}\Bigg] \times \int_{-\infty}^{\phi_j-\widehat Y_{j-1}} \exp\Bigg[\Big(B_N -
{a_{N,j}\over \D_j N}\Big) x\Bigg] {dx \over \sqrt{2\pi \Delta_j N}}. 
\eea \eeq }
Observe that, for $N$ large enough, \hbox{\small{$ B_N -{a_{N,j}\over N\D_j}$}} is strictly positive (it converges to
\hbox{\small{$B-\be_j$}}), whence the existence of the last integral above, which together with the bounds \eqref{evident} leads to
\beq \bea \label{control_three}
\eqref{control_two} \lesssim 2^{-G_j N}\exp\Bigg[ \Big(B_N-{a_{N,j}\over
\D_j N}\Big)(\phi_j -\widehat Y_{j-1})\Bigg]. 
\eea \eeq
Plugging \eqref{control_three} into \eqref{control_one} and iterating the procedure with $B_N$ replaced by \hbox{\small{${a_{N,j}\over
N\D_j} = \be_j + \tilde \epsilon_N$}} (with some new $\tilde \epsilon_N\to 0$) yields the claim.
\end{proof}

For arbitrary $R>0$, let us write $\Sigma_{N,A_j}^R$ for the (random) subset of $\Sigma_{N,A_j}$ such that
 $\overline X_{\tau(1),\dots, \tau(l)} \in [-R, R]$ for every $l\leq j.$ 
\begin{prop}\label{existence_lemma}
Let $\Diamond\subset \R$ be a compact set. Then, to $\vare>0$, we may find large enough $R>0$ such that, 
for large enough $N$, 
\beq \label{localized_r}
\PP\Big[\exists \tau \in \Sigma_{N, A_j}\setminus \Sigma_{N,A_j}^R: \widehat X_{\s(1), \dots, \s(j)}\in \Diamond \Big] \leq \varepsilon,
\eeq
\end{prop}
\begin{proof}
The proof comes in different steps. \\

We first claim that to $\e>0$ there exists $C$ such that 
\beq \label{step_one}
\PP\Big[\exists \tau \in \Sigma_{N, A_j}:\, \widehat X_{\tau(1),\dots, \tau(l)} \geq
C\; \text{for some}\ l\leq j\Big] \leq \epsilon.
\eeq
To see this, we will proceed by induction: suppose that there exists $\widehat C$ such that 
\[\PP\Big[ \forall \tau \in \Sigma_{N,A_l}:\; \widehat X_{\tau(1),\dots, \tau(l)} \leq \widehat C,\; \forall\; l\leq j-1\Big] \geq 1-
\epsilon/2\] for $N$ large enough. For any $\widetilde C > 0$ we thus have
\small{
\beq \bea \label{existence_inizio}
&\PP\Big[ \exists \tau \in \Sigma_{N, A_j}: \widehat X_{\tau(1),\dots, \tau(j)} \geq \widetilde C \Big] \leq {\epsilon \over 2} + \\
& \hspace{2cm} + \PP\Big[\exists \tau \in \Sigma_{N,A_j}:\; \widehat X_{\s(1),\dots, \s(j)} \geq \widetilde C\; \text{and}\; \forall l\leq (j-1) \;\;\widehat
X_{\tau(1),\dots, \tau(l)} \leq \widehat C\Big], 
\eea \eeq}
and the second term on the r.h.s above is bounded by
\small{
\beq \bea \label{existencetwo}
& \sum_{\tau\in \Sigma_{N,A_j}} \PP\Big[ \widehat X_{\tau(1)} \leq \widehat C,\dots, \widehat X_{\tau(1),\dots, \tau(j-1)} \leq \widehat C, \widehat X_{\tau(1),\dots, \tau(j)} \geq \widetilde C\Big]\\
&\quad = 2^{\g(A_j)N}\PP\Big[\widehat Y_1 \leq \widehat C, \dots, \widehat Y_{j-1} \leq \widehat C, \overline Y_j \geq \widetilde C- \widehat Y_{j-1} \Big]\\
&\quad= 2^{\g(A_j)N}\E\Bigg[\int_{\widetilde C- \widehat Y_{j-1}}^{\infty} \exp\Big[ -
{(x+a_{N,j})^2 \over 2\D_j N}\Big] {dx\over \sqrt{2\pi \D_j N}};\; \widehat Y_1 \leq \widehat C, \dots, \widehat Y_{j-1}\leq \widetilde
C\Bigg] \\
&\quad \lesssim 2^{\g(A_j)N}\E\Bigg[ \exp\Big[- {a_{N,j}^2\over \D_j
N} - {a_{N,j}\over 2\D_j N}\big( \widetilde C- \widehat Y_{j-1}\big) +o(1)\Big]; \widehat Y_1 \leq \widehat C, \dots, \widehat Y_{j-1}\leq
\widetilde C\Bigg]\\
&\quad \stackrel{\text{Lemma \ref{control}}}{\lesssim}   \exp\Bigg[ \sum_{l=1}^{j-1} (\be_{l+1}-\be_l) \widehat C - \be_j
\widetilde C\Bigg].
\eea \eeq}
It thus suffices to choose $\widetilde C$ large enough in the positive to make the above less then $\epsilon/2$.  Setting $C\defi
\max\{\widetilde C, \widehat C\}$ yields \eqref{step_one}.\\

We next claim that to $\e>0$ there exists $\widehat R>0$ such that 
\beq\label{step_two}
\PP\Big[\exists \tau \in \Sigma_{N, A_j}: \widehat X_{\tau(1),\dots, \tau(j)} \in \Diamond,\,
\widehat X_{\tau(1),\dots, \tau(l)} \notin [-\widehat R, \widehat R]\,\text{for some}\; l\leq j \Big] \leq \epsilon.
\eeq
Since $\widehat X_{\s(1), \dots, \s(k)}= \overline X_{\s(1),\dots, \s(k)} - \overline X_{\s(1),\dots, \s(k-1)}$ (for $k=2,\dots, j$), \eqref{step_two} would immediately imply \eqref{localized_r}. 

To see \eqref{step_two}, let \hbox{\small{$\widetilde C>0$}} and \hbox{\small{$x_{\Diamond} \defi \sup \{x \in \Diamond \}$}}. By \eqref{step_one} we can find $C>0$ such
that for large enough $N$
\beq \label{existence_three}
\PP\big[\forall \tau \in \Sigma_{N,A_j}: \widehat X_{\tau(1),\dots, \tau(j)} \leq C \,\text{for all}\, l\leq j \big] \geq
1-\epsilon/2.
\eeq 
and therefore
\beq \bea \label{existence_four}
&\PP\Big[\exists \tau \in \Sigma_{N, A_j}:\; \widehat X_{\tau(1),\dots, \tau(j)} \in {\Diamond}, \; \widehat X_{\tau(1),\dots, \tau(l)}
\leq -\widetilde C\; \text{for some}\, l\leq j\Big] \\
&\qquad \leq \epsilon/4 + \PP\Big[\exists \tau \in \Sigma_{N, A_j}:\; \widehat X_{\tau(1),\dots, \tau(j)} \in {\Diamond}, \; \widehat X_{\tau(1),\dots, \tau(l)}
\leq -\widetilde C\\
&\hspace{7cm}\text{for some}\;l\leq j,\quad \widehat X_{\tau(1),\dots, \tau(r)} \leq C \;\forall r\leq j\Big] \\
&\qquad \leq \epsilon/2 + const \times \sum_{l\leq j} \exp\Bigg[ \sum_{k\neq l} (\be_{k+1}-\be_k)\max(C,x_{\Diamond}) -
(\be_{l+1}-\be_l)\widetilde C\Bigg].
\eea \eeq
(the steps behind the last inequality following verbatim those in \eqref{existencetwo}). It thus suffices to choose $\widetilde C$ large enough
in the positive to make \eqref{existence_four} smaller then \hbox{\small{$\epsilon/2$}}, which together with \eqref{existence_three} yields the claim of
\eqref{step_two} with \hbox{\small{$\widehat R =\max(C, \widetilde C)$}}. The Proposition then follows. \\

\end{proof}

We now introduce an important {\it thinning procedure} (the meaning of this wording will become clear below): for $\vare_1>0, \,k=1,\dots, j$ and critical subset $A\subsetneq A_{k}\setminus A_{k-1}$ we say that  ${\bf T}_1(\s,k, A, \varepsilon_1)$ holds if 
\small{\[{1\over \widehat \a_k(A)}\sum_{J\in \widehat {\cal P}_{A, k}} X_{\s_J}^{J} - {1\over \widehat \a_k^c (A)}\sum_{J\in \widehat {\cal P}_{A,k}^c } X_{\s_J}^{J} \leq -
\varepsilon_1 \sqrt{N}.\]} 
We say that ${\bf T_1}(\vare_1)$ holds, tacitly understanding that it holds for all critical subsets. 

\begin{rem} \label{complement_also_critical}
${\bf T}_1$ makes sense only provided the first irreducibility Condition $\bf {\bf c}$ is satisfied, which also guarantees that $\mathcal C_\cdot >0$. In fact, for critical $A\subsetneq A_{l}\setminus A_{l-1}$, by simple
properties of real numbers we also have   
\[ 
\Big[\g(A_l) - \g(A\cup A_{l-1})\Big]\Big/ \widehat \a_l^c(A) = \be_j^2\big/ (2\log 2). 
\]
But by Condition {\bf {\bf c}} there exists $J\in {\cal P}_{A_l} \setminus {\cal P}_{A\cup A_{l-1}}$ with
$J \cap A \neq \emptyset$, in which case $\widehat \a_l^c(A) >\widehat\a_l\big(A_l \setminus (A\cup A_{l-1})\big)$. 
This implies that the relative complement $A_l \setminus (A \cup A_{l-1})$ cannot be critical,  
\[ 
\Big[\g(A_l)- \g(A\cup A_{l-1})\Big] \Big/ \widehat \a_l\Big(A_l \setminus (A\cup A_{l-1})\Big) >
\be_j^2 \big/ 2\log 2. 
\]
To further clarify, consider the example $X_{\s} = X_{\s_1}^{\{1\}}+X_{\s_2}^{\{2\}}$ with parameters $a_1 = a_2 = \g_1 = \g_2 = 1/2$. The associated chain is then
\hbox{\small{${\bf T} = \{A_o = \emptyset,A_1 = \{1,2\}\}$}} and both subsets $\{1\}, \{2\}$ are critical. Evidently, Condition {\bf {\bf c}} is not satisfied. The truncation
${\bf T}_1$ is (to given $\vare$) meaningless since it is fulfilled by those $\s\in  \Sigma_N$ such that  \hbox{\small{$X_{\s_1}^{\{1\}} - X_{\s_2}^{\{2\}} \leq - \vare \sqrt{N}$}} and simultaneously
\hbox{\small{$X_{\s_2}^{\{2\}} - X_{\s_1}^{\{1\}}\leq - \vare \sqrt{N}$}}: there is no such configuration. 
\end{rem}

For technical reasons, we introduce yet another thinning procedure: for $\vare_2>0, k=1,\dots,
j$ and (critical and non critical) subsets $A\subsetneq A_{k}\setminus A_{k-1}$ such that $\widehat \a_k(A)>0$, we say that ${\bf T}_2(\s, k, A, \vare_2)$ holds if 
\[\sum_{J\in \widehat {\cal P}_{A,k}} X_{\s_J}^J \leq \be_k \widehat \a_k(A) (1+ \vare_2) N.\] 
Again, ${\bf T}_2(\vare_2)$ holds, if it holds for all possible subsets.\\

To given $R>0$ we denote by $\Sigma_{N,A_j}^{R, \vare_1, \vare_2}$ the (random) subset of $\Sigma_{N,A_j}^R$ consisting of those configurations which satisfy $\boldsymbol T_1$ and $\boldsymbol T_2$.

\begin{prop} \label{random_configuration_space} Let $R, \vare_2>0$. Then,  
$
\lim_{\vare_1 \downarrow 0} \lim_{N\uparrow\infty} \PP\Big[\Sigma_{N,A_j}^{R} \setminus \Sigma_{N,A_j}^{R, \vare_1,\vare_2} \neq
\emptyset \Big] =0$.
\end{prop}
To prove this we need some additional facts.\\

For compact $\Diamond\subset \R$, we set  $p_N(j, \Diamond) \defi \PP\Big[ \overline Y_j \in \Diamond \Big]$. Let $\vare>0$ and $\eta \in (0, 1/2)$.  For 
critical $A\subsetneq A_{j} \setminus A_{j-1}$  we write
\[ 
p_N(j, \Diamond, A; \vare, \eta) \defi \PP\Big[\overline Y_j \in \Diamond, {Y_{j,A}\over \widehat \a_j(A)} - {Y_{j,A}^c\over \widehat \a_j^c(A)} \geq
-\vare, \sqrt{N} Y_{j,A} - a_{N,j}(A) \leq N^\eta \Big], 
\]
For non-critical $A\subsetneq A_{j}\setminus A_{j-1}$ such that $\widehat \a_j(A)>0$, 
\[
p_N^{>}(\Diamond, j, A, \vare) \defi \PP\Big[\overline Y_j \in \Diamond, Y_{j,A}> \be_j \widehat \a_j(A)(1+\vare) \sqrt{N} \Big]
\]
\begin{lem} \label{quadratic_expansion} For $N$ large enough:
\begin{itemize}
\item[a)] $p_N(j, \Diamond) = 2^{- G_j N}\int_{\Diamond}   \be_j \exp\big[-\be_j x + o(1) \big] dx,$
\item[b)] $p_N^{>}(\Diamond, j, A, \vare) \lesssim 2^{-G_j N}\exp\left[ - const \times \vare^2 N\right].$
\item[c)] $p_N(j, \Diamond, A; \vare, \eta) \lesssim 2^{- G_j N} \times \vare.$
\end{itemize}
\end{lem}
\begin{proof}  Claim $a)$ and $b)$ easily follow from the asymptotics \eqref{evident}. 
To prove c), first recall that $a_{N,j} = a_{N,j}(A) + \be_j \widehat \a_j^c(A) N + O(1)$ and therefore  
\small{ 
\beq \bea \label{starting_critical_truncated}
&p_{N}(j, \Diamond, A; \vare, \eta) \lesssim {1 \over \sqrt{N}}\int_{-\infty}^{N^\eta}
\exp\Big[- {\big(x+ a_{N,j}(A)\big)^2/ 2\widehat \a_j(A) N} \Big] {dx\over \sqrt{2\pi \widehat \a_j(A) N}} \times \\
& \hspace{6cm}\times \int_{\Diamond_x} \exp\Big[- {\big(y+ \be_j \widehat \a_j^c(A) N\big)^2/ 2\widehat \a_j^c(A) N} \Big] dy,\\
& \text{with} \; \Diamond_x \defi \Big\{\Diamond - x + O(1)\Big\}\cap \Big\{ y\in \R:\; {x\over \sqrt{N}\widehat \a_j(A)}
-{y\over \sqrt{N}\widehat \a_j^c(A)} \geq - \vare + O(\log N/\sqrt{N})\Big\}.
\eea \eeq}
Since $\Diamond$ is bounded, for the integration set $\Diamond_x$ not to be empty we must have  $x \geq x_{min} \defi - const \cdot \vare \cdot \sqrt{N} +  O(\log N)$, with
$const = \widehat \a_j(A) \widehat \a_j^c(A) \big/ \D_j$. Therefore: 
\small{ 
\beq \bea \label{bound_critical_eta_two}
\eqref{starting_critical_truncated} &\lesssim { 1\over \sqrt{N}} \exp\Big[- {\be_j^2 \over 2}\widehat \a_j^c(A)N \Big] \int_{\Diamond} \exp\big[-\be_j y \big] dy 
\times \\
& \hspace{3cm} \times \int_{x_{min}}^{N^\eta} \exp\big( \be_j x\big) \exp\Bigg[- {\big(x+ a_{N,j}(A)\big)^2\over 2\widehat \a_j(A)N} \Bigg] {dx
\over \sqrt{2\pi \widehat \a_j(A)N}} \\
&\lesssim { 1\over \sqrt{N}} \exp\Big[- {\be_j^2 \over 2}\widehat \a_j^c(A)N  + {\be_j^2\over 2} \widehat \a_j(A) N - a_{N,j}(A)\be_j \Big] \times \\
& \hspace{3cm} \times \int_{x_{min}}^{N^\eta} \exp\Bigg[-{\big(x+ a_{N,j}(A) - \be_j \widehat \a_j(A) N \big)^2\over 2\widehat \a_j(A)N} \Bigg] {dx
\over \sqrt{2\pi \widehat \a_j(A)N}}\\
& \lesssim 2^{-G_j N} \times \PP\Big[Y_{j, A} \in \Big(x_{min} N^{-1/2}, N^{\eta -1/2}\Big) + O(\log N/\sqrt N)\Big]
\eea \eeq }
the last step by simply noting that $a_{N,j}(A) - \be_j \widehat \a_j(A) N = O(\log N)$. Remark that
\[ 
\lim_{N\to \infty} \PP\Bigg[Y_{j, A} \in \Big(x_{min} N^{-1/2}, N^{\eta -1/2}\Big) + O(\log N/\sqrt N)\Bigg] = \int_{-const\cdot \vare}^0 \exp\big(-{x^2\over
2}\big) {dx\over \sqrt{2\pi}} \lesssim \vare.
\]
This settles claim c). 
\end{proof}

\noi {\bf Proof of Proposition \ref{random_configuration_space}}
Since $R$ is fixed throughout the proof, we abbreviate $\Diamond \defi [-R,R]$.
\small{\beq \bea
& \PP\Big[ \Sigma_{N,A_j}^R \setminus \Sigma_{N,A_j}^{R, \vare_1, \vare_2} \neq \emptyset \Big] \\
& \quad \leq  \PP\Bigg[ \exists \s\in \Sigma_{N,A_j}^R: \sum_{J\in \widehat{\cal P}_{A,k}} X_{\s_J}^J - a_{N,k}(A) \geq N^{\eta}\;\\
&\hspace{6cm} \text{for some critical} \; A\subsetneq A_k \setminus A_{k-1}, k =1,\dots, j\Bigg] + \\
& \qquad + \PP\Bigg[ \exists \s\in \Sigma_{N, A_j}^R, \; {\bf T}_1(\s, k, A, \vare_1)\; \text{does not hold for critical}\; A\subsetneq A_{k}\setminus A_{k-1}\\
& \hspace{5cm}  \text{for some}\; k = 1,\dots,j \; \text{but}\;
\sum_{J\in \widehat {\cal P}_{A,k}} X_{\s_J}^J - a_{N,k}(A) \leq N^{\eta}\Bigg] +  \\
&  \qquad + \PP\Bigg[ \exists \s\in \Sigma_{N,A_j}^R\; \text{such that}\; {\bf T_2}(\s,k, A, \vare_2) \; \text{does not hold for some}\\
&\hspace{6cm} \text{for some} \; A\subsetneq A_k \setminus A_{k-1}, k =1,\dots, j\Bigg] \\
& \quad = (I) + (II) + (III).
\eea \eeq}
We provide upper-bounds to the three different terms on the r.h.s above. 
\beq \bea \label{bound_(I)}
(I) &\leq \sum_{k=1}^j \mathop{\sum_{A\subsetneq A_{k}\setminus A_{k-1}}}_{A\;\text{critical}} \PP\Bigg[ \exists \s \in
\Sigma_{N,A_{k-1}\cup A}, \; \text{such that}\\
& \hspace{2cm} \forall l=1, \dots, k-1\; \overline X_{\s(1), \dots, \s(l)} \in \Diamond, \sum_{J\in \widehat {\cal P}_{N,A_k}}
X_{\s_{J}}^{\{J\}} -a_{N,k}(A) \geq N^\eta\Bigg]\\
& \leq \sum_{k=1}^j \mathop{\sum_{A \subsetneq A_{k} \setminus A_{k-1}}}_{A\;\text{critical}} 2^{\g(A_{k-1})N}\left\{\prod_{l=1,\dots, k-1}
p_N(l, \Diamond) \right\} 2^{\g(A)N} \PP\Big[ \sqrt{N} Y_{k, A} -a_{N,k}(A) \geq N^\eta \Big]
\eea \eeq
It is easily seen that $\PP\Big[ \sqrt{N} Y_{k, A} -a_{N,k}(A) \geq N^\eta \Big] \lesssim \exp\Big[-{\be_k^2\over 2}\widehat \a_k(A)N-const \times N^\eta \Big]$
for some positive $const$, and for critical $A\subsetneq A_{k}\setminus A_{k-1},\;{\be_k^2\over 2 }\widehat \a_k(A) = \g(A) \log 2$, so it follows from 
Lemma \ref{quadratic_expansion} that  $(I) \lesssim \exp\big[-const\times N^{\eta} \big]$ for large enough $N$.  
\beq \bea \label{bound_critical_eta}
& (II) \leq  \sum_{\s \in \Sigma_{N,A_j}}  \mathop{\sum_{k=1\dots, j}}_{A\subsetneq A_{k}\setminus A_{k-1}\; \text{critical}}\PP\Bigg[\overline X_{\s(1), \dots,
\s(l)} \in \Diamond, l\leq k, \; {\bf T_1}(\s,k, A, \vare_1)\; \text{holds}, \\
& \hspace{9cm} \; \sum_{J\in {\widehat \cal P}_{A,k}} X_{\s_J}^{\{J\}} - a_{N,k}(A) \leq N^\eta\Bigg]\\
& \qquad \leq 2^{\g(A_j) N}\mathop{\sum_{k=1,\dots, j}}_{A\subsetneq A_{k}\setminus A_{k-1}\;\text{critical}} p_N(k, \Diamond;
\vare_1, \eta)\times \mathop{\prod_{l=1,\dots, j}}_{l\neq k} p_N(\Diamond,
l)
\eea \eeq
Hence, by Lemma \ref{quadratic_expansion}, we have $(II)  \lesssim \vare_1$ for large enough $N$. Finally, 
\beq \bea 
(III) &\leq \mathop{\sum_{k=1,\dots, j}}_{A\subset A_{k}\setminus A_{k-1}} 2^{\g(A_{k})N} p_N^{>}(\Diamond, k, A, \vare_2) \prod_{l=1\dots, k-1} p_N(l, \Diamond) 
\eea \eeq
which by Lemma \ref{quadratic_expansion} is easily seen to be $\lesssim \exp[ - const \times \vare_2 ^2\times N]$ for some positive $const >0$. 
Putting the pieces together, we see that $\PP\Big[ \Sigma_{N,A_j}^R\setminus \Sigma_{N,A_j}^{R, \vare_1, \vare_2}\Big] = o(\vare_1)$. 
\begin{flushright}
$\square$
\end{flushright}

\subsection{Suppression of structures and propagation} \label{proof_non_ultrametric}
We first derive some bounds on "two-points probabilities". Let
\[\bea
&p_{N}^{(2)}(j, {\Diamond}, A,\varepsilon) \defi \PP\Big[\sqrt{N}Y_{j,A}+ \sqrt{N} Y_{j, A}^c-a_{N,j} \in {\Diamond}, \\
& \hspace{4cm} \sqrt{N} Y_{j,A}+ \sqrt{N} Z_{j, A}^c-a_{N,j} \in {\Diamond}, Y_{j,A} \leq \be_j \widehat
\a_{j}(A)(1+\varepsilon) \sqrt{N} \Big],    
\eea \]
and for critical $A \subsetneq A_j \setminus A_{j-1}$ write
\small{
\[\bea
&p_N^{(2,crit)}(j, {\Diamond}, A, \varepsilon) \defi \PP\Bigg[\sqrt{N}Y_{j,A}+ \sqrt{N} Y_{j, A}^c-a_{N,j}\,\text{and}\, \sqrt{N} Y_{j,A}+ \sqrt{N} Z_{j,
A}^c-a_{N,j} \in {\Diamond},\\
&\hspace{6cm}\text{and}\,{Y_{j,A}\over \widehat \a_j(A)} - {Y_{j, A}^c \over \widehat \a_j^c(A)} \leq -\varepsilon, \; {Y_{j,A}\over \widehat
\a_j(A)} - {Z_{j, A}^c \over \widehat \a_j^c(A)} \leq -\varepsilon
\Bigg]
\eea \]}
 \begin{lem} \label{quadratic_expansion_two} Let $\vare>0$. For $N$ large enough 
\begin{itemize} 
\item[a)] 
$ p_N^{(2)}(j, {\Diamond}, A,\varepsilon) \lesssim 2^{-2 G_j N}\exp\Big\{\be_j^2 \widehat \a_j(A)\big[1-{1\over
2}(1-\varepsilon)^2 \big]N\Big\}.$
\item[b)] $ p_{N}^{(2,crit)}(j, {\Diamond}, A, \varepsilon) \lesssim 2^{- 2 G_j N + \g(A)N} \exp\big[ -const \times \varepsilon 
\sqrt{N}\big].$
\end{itemize}
\end{lem}
\begin{proof} $a)$ is straightforward. $b)$ Setting $\omega_N = O(\log N)$ for $N\uparrow \infty$, it holds: 
{\small
\beq \bea \label{critical_one}
&p_{N}^{(2,crit)}(j, {\Diamond}, A, \varepsilon) \lesssim \int_{-\infty}^{\infty} \exp\left[- {\left( x+ a_{N,j}(A)\right)^2 \over 2\widehat
\a_j(A) N} \right] dx \left(\int_{\Diamond_x} \exp\left[- {\left( y+ \be_j \widehat \a^c_j(A)N\right)^2\over 2\widehat \a^c_j(A) N} \right] dy\right)^2, \\
&\quad \text{where}\quad {\Diamond}_x =  \Big\{\Diamond - x - \omega_N \Big\} \bigcap \left\{y\in \R:\; y \geq {\widehat \a_j^c(A)\over \widehat
\a_j(A)} x + \varepsilon \widehat \a_j^c(A) \sqrt{N} + \omega_N \right\}.\eea \eeq}
$\Diamond_x$ is not empty as soon as \hbox{\small{$x \leq x_{\max} \defi - \varepsilon {\widehat \a_j(A) \widehat \a_j^c(A) \over \Delta_j} \sqrt{N} +\omega_N$}}. Thus, 
\small{
\beq \bea \label{critical_two}
\eqref{critical_one} &\lesssim \int_{-\infty}^{x_{\max}} \exp\left[- {\left( x+ a_{N,j}(A)\right)^2 \over 2\widehat
\a_j(A) N} \right] dx \left(\int_{\Diamond - x - \omega_N} \exp\left[- {\left( y+ \be_j \widehat \a^c_j(A)\right)^2\over 2\widehat \a_j^c(A) N} \right] dy \right)^2 \\
& \lesssim \exp\Big[ - {\be_j^2} \widehat \a_j^c(A) N + \omega_N \Big] \int_{-\infty}^{x_{\max}} \exp\Bigg[- { \big(x - \be_j \widehat \a_j(A) N +
\omega_N\big)^2\over 2\widehat \a_j(A)N}\Bigg]dx  \\
& \lesssim  \exp\Big[ - {\be_j^2} \widehat \a_j^c(A) N - {\be_j^2\over 2} \widehat \a_j(A) N + \omega_N\Big]
\underbrace{\int_{-\infty}^{x_{\max}} \exp\big[\be_j x \big] dx}_{\leq \exp\big(- const \times \varepsilon \sqrt{N}\big)}.  
\eea \eeq}
By criticality (cfr. remark \ref{complement_also_critical}), 
\[
{\be_j^2 \over 2} \widehat \a_j(A) = \g(A) \log 2,\quad {\be_j^2\over 2} \widehat \a_j^c(A) = \big[\g(A_j) - \g(A\cup A_{j-1}) \big]\log 2 ,
\]
hence  
\[\eqref{critical_two} \leq 2^{- 2 G_j N} \exp\big[\g(A) N\log 2 \big] \exp\big[ - const \times \varepsilon  \sqrt{N}\big].\]
\end{proof}

We put on rigorous ground the suppression of structures at given level, say $j$.

\begin{prop}[Suppression]\label{onset_um_one} Let $\s',\tau'$ be two reference configurations in $\Sigma_{N,A_{j-1}}$.  For positive $\vare_1$ and
sufficiently small $\vare_2$ there exists $const>0$ such that 
\small{\beq \bea \label{existence_conditional}
&\PP\Big[ \exists\, \s,\tau \in \Sigma_{N, A_j}^{R, \vare_1,\vare_2},\;\s(j)\neq \tau(j), \; \s_{A_{j-1}} =\s', \tau_{A_{j-1}} =
\tau':\\
&\hspace{5cm} \s_s = \tau_s \;\text{for some}\; s \in A_{j}\setminus A_{j-1}\Big] \lesssim \exp\Big[ - const \times \vare_1 \sqrt{N}\Big].
\eea \eeq}
\end{prop}
\begin{proof} 
The l.h.s of \eqref{existence_conditional} is clearly bounded by 
\small{
\beq \bea \label{upper_bound_dependency}
&\mathop{\sum_{A\subsetneq A_{j} \setminus A_{j-1}}}_{A\; \text{critical}} \sum^{\star} \; \PP\Bigg[ \overline X_{\s(1),\dots, \s(j)}\,\text{and}\,\overline X_{\tau(1),\dots, \tau(j)} \in {\mathfrak
R},\;  {\bf T}_1(\s,j, A,\varepsilon_1),\; {\bf T}_1(\tau, j, A, \varepsilon_1)\; \text{hold} \Bigg]+ \\
&+\mathop{\sum_{A\subset A_{j}\setminus A_{j-1}}}_{ A\;\text{non-critical}} \sum^{\star} \; \PP\Bigg[ \overline X_{\s(1),\dots, \s(j)}\,\text{and} \,\overline X_{\tau(1),\dots, \tau(j)} \in {\mathfrak
R};\, {\bf T}_2(\s,j, A,\varepsilon_2),\,\text{and}\, {\bf T}_2(\tau, j, A, \varepsilon_2) \; \text{hold} \Bigg]. \\
\eea \eeq }
In both cases, \hbox{\small{$\stackrel{\star}{\sum}$}} runs over all the \hbox{\small{$\s, \tau \in \Sigma_{N, A_j}$}} such that \hbox{\small{$\s(j) \neq \tau(j)$}}, as well as
\hbox{\small{$ \s_{A_{j-1}} = \s', \tau_{A_{j-1}}=\tau',\; \s_J =\tau_J$}} for every \hbox{\small{$J\in \widehat {\cal P}_{A,j}$ and $\s_J\neq \tau_J$}} for every
\hbox{\small{$J\in \widehat {\cal P}^c_{A,j}$}}. To fixed $A\subset A_{j} \setminus A_{j-1}$ there are at most \hbox{\small{$2^{2G_j N} 2^{-\g(A)N}$}} couples of
$\s,\tau$ satisfying these requirements. Thus we may upper bound \eqref{upper_bound_dependency} by
\small{
\beq \bea \label{final_upper_bound_dep}
&\mathop{\sum_{A\subsetneq A_{j}\setminus A_{j-1}}}_{A\;\text{critical}} 2^{2G_j N} 2^{-\g(A)N} p_N^{(2,crit)}(j,{\Diamond}, A, \varepsilon_1) +\mathop{\sum_{A\subset A_{j}\setminus A_{j-1}}}_{A\;\text{non-critical}} 2^{2G_j N} 2^{-\g(A)N} p_N^{(2)}(j,{\Diamond}, A,
\varepsilon_2)\\
& \stackrel{\text{Lemma}\, \ref{quadratic_expansion_two}}{\lesssim} \mathop{\sum_{A\subset A_{j} \setminus A_{j-1}}}_{A\; \text{critical}} e^{- const \times 
\varepsilon_1 \sqrt{N}} + \mathop{\sum_{A\subset A_{j}\setminus A_{j-1}}}_{A\; \text{non-critical}} 2^{- \g(A)N}\exp\Big\{\be^2_j \widehat
\a_j(A)\Big[1-{1\over 2}(1-\varepsilon_2)^2\Big]N\Big\}. 
\eea \eeq }
For non-critical \hbox{\small{$A, \;\be_j^2 \widehat \a_j(A) < \g(A) 2\log 2$}} strictly, so we can find $\vare_2$ small enough such that 
\beq \label{small_enough_one}
\de'(\vare_1) \defi\max_{j\leq K} \; \max_{A\subsetneq A_{j}\setminus A_{j-1};\; A\; \text{non-critical}}\; \Bigg\{\be^2_j \widehat
\a_j(A)\Big[1-{1\over 2}(1-\vare_2)^2\Big] - \g(A) \log 2\Bigg\} <0.  
\eeq
The second sum on the r.h.s of \eqref{final_upper_bound_dep} is thus \hbox{\small{$\lesssim \exp\big[- |\de'| N\big]$}}, while the first sum is \hbox{\small{$\lesssim \exp\big[
- const \times \varepsilon_1 \sqrt{N}\big]$}}. This proves the claim.
\end{proof}

Suppose now that two configurations $\s,\tau \in \Sigma_{N,A_j}^{R, \varepsilon_1, \vare_2}$ are such that $\s_s =\tau_s$ for some $s\in A_{m} \setminus A_{m-1}$ for some $m\leq j$ but $\s_t \neq \tau_t$ for some $t\in A_{r}\setminus A_{r-1}$ and
$r<m$. Without loss of generality we may assume that there are numbers $k,l,m,\, 0\leq k < l < m \leq j$ such that $\s_{A_{k}} = \tau_{A_{k}}, \,\s_{r}
\neq \tau_{r}\;\forall r \in A_{l}\setminus A_{k}$, and $\s_{A_{m}\setminus A_{l}} = \tau_{A_{m}\setminus A_{l}}$. 

\begin{prop}[Propagation] \label{onset_um_two}
For positive $\vare_1$ and small enough $\vare_2$ there exists positive $const$ such that
\small{\beq \bea \label{equal_not_equal}
&\PP\Bigg[\exists\,\s,\tau \in \Sigma_{N,A_{m}}^{R,\vare_1, \vare_2}:\s_{A_{k}} = \tau_{A_{k}}, \s_r \neq \tau_r \, \forall r\in
A_k\setminus A_l,\; \s_{A_{m}\setminus A_{l}} = \tau_{A_{m}\setminus A_{l}}\Bigg] \lesssim e^{-const\times N}.
\eea \eeq } 
\end{prop}
\begin{proof}
Without loss of generality we may assume $m = l+1$. Consider two configurations $\s,\tau \in
\Sigma_{N,A_{l+1}}$ which differ on the whole \hbox{\small{$A_{l}\setminus A_{k}$}} but \hbox{\small{$\s_{A_{l+1}\setminus A_l} =
\tau_{A_{l+1}\setminus A_l}$}}. By the irreducibility condition
{$\bf c'$} there exists \hbox{\small{$J\in {\cal P}_{A_{l+1}}\setminus {\cal P}_{A_{l}}$}} such that \hbox{\small{$\s_{J} \neq
\tau_{J}$}} in which case there must be a strict subset $A\subsetneq A_{l+1}\setminus A_{l}$ such that \hbox{\small{$\s_J = \tau_J$
for all $J \in \widehat{\cal P}_{l+1, A}$}} and \hbox{\small{$\s_J\neq \tau_J$}} for all \hbox{\small{$J \in \widehat {\cal P}^c_{l+1,
A}$}} (loosely speaking, the associated random variables \hbox{\small{$\overline
X_{\s(1),\dots,\s(l+1)}$}} and \hbox{\small{$\overline X_{\tau(1),\dots, \tau(l+1)}$}} cannot coincide). We can therefore bound the l.h.s. of \eqref{equal_not_equal} by 
\small{
\beq \bea \label{firstbound}
& \mathop{\sum_{A\subsetneq A_{l+1} \setminus A_{l} }}_{A\; \text{critical}} \sum^* \PP\Big[\overline X_{\s(1),\dots, \s(j)}\;\text{and}\; \overline X_{\tau(1),\dots,
\tau(j)} \in {\Diamond} \quad\text{for all}\; j=1, \dots k,\dots,l+1;  \\
&\hspace{8cm} {\bf T}_1(\s,l, A, \varepsilon_1)\; \text{and}\; {\bf T}_1(\tau,l, A, \varepsilon_1) \; \text{hold} \Big] + \\
& \mathop{\sum_{A\subsetneq A_{l+1} \setminus A_{l}}}_{A\; \text{non-critical}} \sum^* \PP\Big[\overline X_{\s(1),\dots, \s(j)}\;\text{and}\; \overline X_{\tau(1),\dots,
\tau(j)} \in {\Diamond} \quad \text{for all}\;j=1, \dots l+1;  \\
&\hspace{7cm} {\bf T}_2(\s,l, A, \varepsilon_2)\; \text{and}\; {\bf T}_2(\tau,l, A, \varepsilon_2) \; \text{hold}\Big],
\eea \eeq}
where $\stackrel{\star}{\sum}$ runs over those \hbox{\small{$\s,\tau$}} in \hbox{\small{$\Sigma_{N, A_{l+1}}$}} such that $\s_J = \tau_J$
for all \hbox{\small{$J\in \widehat {\cal P}_{l+1, A}$}}, \hbox{\small{$\s_J \neq \tau_J$}} \hbox{\small{$J\in \widehat {\cal
P}_{l+1,A}^c$}}, \hbox{\small{$\s_{A_k} = \tau_{A_k}$}}, \hbox{\small{$\s_{s}\neq \tau_s\, \forall s\in A_{l}\setminus A_k$}}, \hbox{\small{$\s_{A_{l+1}\setminus A_l}=
\tau_{A_{l+1}\setminus A_{l}}$}}. \\

\noi We also observe that \hbox{\small{$\s_s \neq \tau_s \;\text{for all}\; s\in A_{l}\setminus A_{k}$}}
implies that the random variables \hbox{\small{$\overline X_{\s(1),\dots, \s(j)}$}} and \hbox{\small{$\overline
X_{\tau(1),\dots,\tau(j)}$}} are independent for all $j=k+1\dots l$. In fact, for every \hbox{\small{$J\in {\cal P}_{A_{l}}\setminus {\cal
P}_{A_{k}}$}} by construction \hbox{\small{$J\cap (A_{l}\setminus A_{k}) \neq \emptyset$}}; this amounts to say that for every such $J$ there exists at least one \hbox{\small{$s\in
A_{l}\setminus A_{k}$}} with $J\ni s$. \\

\noi The above remarks, together with some simple counting steadily yield
\small{
\beq \bea \label{secondbound}
&\eqref{firstbound} \lesssim 2^{N\big[\g(A_{k})+2\g(A_{l}\setminus A_{k})+\g(A_{l+1}\setminus A_l)\big]}\prod_{r\leq k} p_N\big(r,{\Diamond}\big) \prod_{r=k+1}^{l}
p_N\big(r, {\Diamond}\big)^2  \times \\
&\hspace{2cm} \times  \Bigg\{ \mathop{\sum_{A\subsetneq A_{l+1}\setminus A_l}}_{A\; \text{critical}} p_{N}^{(2,crit)}(l+1, {\Diamond}, A, \varepsilon_1)
+\mathop{\sum_{A\subsetneq A_{l+1}\setminus A_{l}}}_{A\;\text{non-critical}} p_N^{(2)}\big(l+1, {\Diamond}, A, \varepsilon_2\big) \Bigg\} \\
&\stackrel{\text{Lemma}\;\ref{quadratic_expansion_two}}{\lesssim} \mathop{\sum_{A\subsetneq A_{l+1}\setminus A_{l}}}_{A\; \text{non-critical}} \exp\left\{ 2\log 2 G_{l+1} N \left[ \Big(1- {1\over 2}\big(1-\varepsilon_2\big)^2\Big) {\widehat
\a_{l+1}(A)\over \Delta_{l+1}} -{1\over 2} \right] \right\} + \\
& \hspace{5cm} + \mathop{\sum_{A\subsetneq A_{l+1}\setminus A_{l}}}_{A\; \text{ critical}}  2^{(\g(A) - G_{l+1})N}
\exp\big[- const \times \varepsilon_1 \sqrt{N}\big].
\eea \eeq}
Clearly, the second sum on the r.h.s above is \hbox{\small{$\lesssim \exp\big[-\big|\de'\big| N\big]$}} for 
\[
\de' \defi \max_{l\leq K-1} \mathop{\max_{A\subsetneq A_{l+1} \setminus A_l}}_{A\; \text{critical}} \Big\{ \g(A) - G_{l+1} \Big\} <0.
\] 
It is crucial that the first sum runs over (non-critical) subsets strictly included in $A_l\setminus A_{l+1}$, since it guarantees that $\max_{A\subsetneq
A_{l+1}\setminus A_l} \widehat \a_{l+1}(A) < \Delta_{l+1}$  and thus, for small enough $\vare_2$, 
\beq \label{small_enough_two}
\de''(\varepsilon_2) \defi \max_{l\leq K-1}\max_{A\subsetneq A_{l+1}\setminus A_l} \Bigg\{ (2\log 2) G_{l+1}\Bigg[\Big(1- {1\over
2}\big(1-\vare_2\big)^2\Big){\widehat \a_{l+1}(A)\over \Delta_{l+1}}
-{1\over 2} \Bigg] \Bigg\} <0.
\eeq
This settles the Lemma with $const \defi \min\big\{ \big|\de'\big|, \big|\de''\big|\big\}$. 
\end{proof}

\subsection{Proof of Proposition \ref{nonultra_prop}}
Let $\e >0$ and the compat set $\Diamond \subset \R$ be given. By Proposition \ref{localized_r} and \ref{random_configuration_space} we may find $R>0$ and $\vare_1 >0$, such that (for any $\vare_2$)
\[
\PP\left[\exists \s \in \Sigma_{N, A_j}\setminus \Sigma_{N, A_j}^{R, \vare_1, \vare_2}: \widehat X_{\s(1), \dots, \s(j)} \in \Diamond \right]\leq \e/3, 
\]
for large enough $N$. 

By Markov inequality, together with the estimates from Lemma
\ref{quadratic_expansion}, it is easily seen that there exists ${\sf N} = {\sf N}(\epsilon)$ such that the probability that there exist
more than $\sf N$ configurations in \hbox{\small{$\Sigma_{N,A_j}^{R, \vare_1, \vare_2}$}} is smaller than
$\epsilon/3$.

Therefore, it suffices to estimate the probability that, out of a finite number $\sf N$ of configurations in \hbox{\small{$\Sigma_{N,A_j}^{R, \vare, \vare_2}$}} some of them form a non ultrametric
couple. But this case is taken care of by Proposition \ref{onset_um_one} and \ref{onset_um_two} (and a straightforward combination of the two). By choosing $\vare_2$ {\it small
enough}, in the range of validity of \eqref{small_enough_one} and \eqref{small_enough_two}, the probability of such an event is of order $\exp[-
const \times \sqrt{N}]$, thus smaller than $\epsilon/3$ for large enough $N$.

This settles the claim. 
\begin{flushright}
$\square$
\end{flushright}

\subsection{Proof of Proposition \ref{partial_energies}} \label{proof_extremal_process_overlap}
Let $R, \vare_1, \vare_2$ be given, and consider the element ${\mathcal N}_{j,N}^{R,\vare_1,\vare_2}$ of $\mathcal M_{mp}(\R^{2}\times 2^{A_j})$ induced naturally by the collection 
$(\widehat X_{\s(1), \dots, \s(j)}, \s \in \Sigma_{N, A_j}^{R,\vare_1, \vare_2})$. We denote by $\widehat X_{j, N}^{R, \vare_1, \vare_2}$ the law of such a process. 
We now claim that in order to prove Proposition \ref{partial_energies} it suffices to prove that for $\vare_2$ in the range of validity of \eqref{small_enough_one} and \eqref{small_enough_two}, 
\beq \label{approx_by_truncation} 
\lim_{N\to \infty} \widehat X_{j, N}^{R, \vare_1, \vare_2} = \widehat X_{j}^{R, \vare_1},
\eeq
where the latter is the law of the element in $\mathcal M_{mp}(\R^{2}\times 2^{A_j})$ naturally induced by the collection of points $(x_{\boldsymbol i}, \boldsymbol i \in \N^j)$, with $x_{\boldsymbol i} = x_{i_1}^1+\dots + x_{i_1,\dots, i_j}^j$ and the properties: 
{\it i.} For $l=1,\dots, j$ and multi-index ${\bf i}_{l-1}$, the point process \hbox{\small{$(x_{{\bf i}_{l-1},
i_l}^l;\; i_l\in \N)$}} is poissonian with density \hbox{\small{${\mathcal C}^{\vare_1}_{l} \cdot \be_l e^{-\be_l t} dt$}} on $[-R,R]$ (and zero otherwise). {\it ii.} The $x^l$ are independent for different $l$.
{\it iii.} $(x_{{\bf i}_{l-1}, i_l}^l;\; i_l \in \N)$ are independent for different ${\bf i}_{l-1}$. {\it iv.} If \hbox{\small{$A_{l}\setminus A_{l-1}$}} contains no critical subsets, then ${\mathcal C}^{\vare_1}_{l} = 1$, otherwise   
\[ 
{\mathcal C}_{l,\vare_1} = \PP\left[\left\{{Y_{l, A}\over \widehat \a_l(A)} - {Y_{l, A}^c\over \widehat \a^c_l(A) }\leq -\vare_1 \right\} \; \forall A \subsetneq A_{l}
\setminus A_{l-1}, \; A\;\text{is critical} \right]. 
\]
In fact, it is rather straightforward that, with $\widehat X_j$ as in Proposition \ref{partial_energies},  
\beq 
\lim_{\vare_1 \to 0} \lim_{R\to \infty} \widehat X_{j}^{R, \vare_1} = \widehat X_{j},
\eeq
and therefore, by Proposition \ref{existence_lemma} and \ref{random_configuration_space}, \eqref{approx_by_truncation} would automatically imply Proposition \ref{partial_energies}. \\

So, the crucial step to prove Proposition \ref{partial_energies} is really to prove \eqref{approx_by_truncation}.\\

The underlying Derrida-Ruelle cascades enjoy important properties that we will exploit in order to get \eqref{approx_by_truncation}. Most importantly, once one knows what happens on level $j-1$ (the distribution on the real axis of the
points $x_{i_1}^{1}+\dots + x^{j-1}_{i_1, \dots, i_{j-1}}$, as well as their overlap structure) the "full process" is obtained by adding random points independently: conditioned on the 
first $j-1$ levels, given $k\in \N$ multi-indeces ${\bf i}^1, \dots, {\bf i}^k \in \N^{j-1}$, and disjoints $B_1, \dots, B_k \subset [-R,R]$, 
we have the following equality in distribution 
\beq \label{fundamental_form} 
\left(\sum_{l\in \N} \de_{x^j_{{\bf i}^1, l}}(B_1), \dots, \sum_{l\in \N}\de_{x^j_{{\bf i}^k, l}}(B_k)\right) \stackrel{(d)}{=} \Big(V_1, \dots, V_k\Big)
\eeq 
with the random variables $V_r, r=1,\dots, k$ being independent, Poisson-distributed of parameters $\mu_{\vare_1}(B_r) \defi \int_{B_r} C_{j, \vare_1} \be_j e^{-\be_j t} dt$.  
By conditioning, the finite dimensional distribution of the limiting process $\widehat X_{j}^{R, \vare_1}$ can be brought back to
expressions such as \eqref{fundamental_form}, and in fact the same line of reasoning works also for the finite $N$ system, as we shall elucidate below.\\ 

We introduce the projection $\mathfrak P: \R^2 \to \R, (x, y)\mapsto x+y$, and consider the points
\[
\Bigg\{\Big(\widehat X_{\s(1), \dots, \s(j-1)}, \overline X_{\s(1), \dots, \s(j-1), \s(j)}\Big), \s\in \Sigma_{N,A_j}^{R, \vare_1, \vare_2}.\Bigg\}
\]
This induces naturally a process \hbox{\small{${\cal N}^{(2)}_{N,j} \in {\mathcal M}_{mp}\left( (\R^2)^{(2)} \times 2^{A_j}\right)$}}, where, to lighten notations we omit the dependence on $R, \vare_1, \vare_2$. The process ${\cal N}_{N,j}^{R,\vare_1,\vare_2}$ is then the "image" of ${\cal
N}_{N, j}^{(2)}$ under the projection $\mathfrak P$ (the points $(\widehat X_{\s(1), \dots, \s(j-1)}, \overline X_{\s(1), \dots, \s(j)})$ are projected to
$\widehat X_{\s(1), \dots, \s(j-1)} + \overline  X_{\s(1), \dots, \s(j)} =  \widehat X_{\s(1), \dots, \s(j)}$). To handle the 
finite dimensional distributions of the "multidimensional process" ${\mathcal
N}^{(2)}_{N,j}$, we observe that is easily follows from Proposition \ref{nonultra_prop} that
\[ 
\lim_{N\to \infty} \PP\Big[{\cal N}_{N,j}^{(2)}(\R\times \R; A)>0 \Big] = 0, \;\forall A\in 2^{A_j} \setminus \{ \emptyset, A_1, \dots, A_{j-1} \}.
\]
The events involving overlaps in the chain $\{\emptyset, \dots, A_j\}$ are easily handled through the following remark: conditionining the process ${\cal N}_{N,j}^{(2)}$ to the 
sigma-field generated by the process ${\cal N}_{N,j-1}$ amounts to prescribe a finite number, say $L$, of configurations $\s^1, \dots, \s^L\in \Sigma_{N,A_{j-1}}$, as well
as their overlap structure. By ultrametricity, the overlaps among these $L$ configurations take values in the chain $\{\emptyset, \dots, A_{j-1}\}$ only. But then, it is
easy to reformulate the finite dimensional distributions of the process ${\cal N}_{N,j}^{(2)}$ given the process ${\cal N}_{N,
j-1}^{R, \vare_1, \vare_2}$ into finite dimensional probabilities of the point processes $(\overline X_{\s^r, \tau}, \tau \in \Sigma_{N,A_j \setminus A_{j-1}})$, with prescribed $\s^1, \dots,
\s^L$ for $r = 1,\dots, L$. Summarizing, one gets the weak convergence of ${\cal N}_{N,j}^{(2)}$ towards the process ${\cal N}^{(2)}_j$ naturally induced by the points $\big\{(y_{\bf i},
y_{{\bf i}, l}); {\bf i}\in \N^{j-1}, l\in \N \big\}$ on $\R^2$, and (by continuity on compacts of the projection ${\mathfrak P}$) weak convergence of ${\cal N}_{N,j}^{R,\vare_1,\vare_2}$ as soon as we prove that for given family of reference configurations $\s^{1}, \dots, \s^{k} \in
\Sigma_{N,A_{j-1}}$ with a certain overlap structure $q(\s^r, \s^t) \in \{\emptyset, \dots, A_{j-1}\}$, and $r,s=1,\dots, k$ the distribution of the random vector 
\beq \label{multivariate}
\left( \sum^{(1)} \de_{\overline X_{\s(1), \dots, \s(j)}}(B_1), \dots, \sum^{(k)} \de_{\overline X_{\s(1), \dots, \s(j)}}(B_k)\right)
\eeq
(with sums running over those $\s\in \Sigma_{N,A_j}$ such that $\s_{A_{j-1}} = \s^r$ and satisfying conditions ${\bf T_1}(\vare_1)$ and ${\bf T_2}(\vare_2)$ on $j$-level) 
is approximately multivariate Poisson, cfr. \eqref{fundamental_form}. To see this last step, we will use the socalled {\it Chen-Stein method}, a particularly efficient tool in Poisson approximation, cfr. \cite{barbour}. \\ 

We begin with a technical estimate. For bounded real subset $\Diamond$, and $\de, \rho>0$ we set:   
\[ \bea
&p_N^{\de, \rho}\big(j, {\Diamond}\big) \defi \PP\Bigg[\overline Y_j \in {\Diamond};\;\forall\; \text{critical}\;B\subsetneq A_{j}\setminus A_{j-1}: \; {Y_{j,B}\over \widehat
\a_j(B)} - {Y_{j, B}^c \over \widehat \a_j^c(B)} \leq -\de; \\
&\hspace{4cm} \forall\; A\subset A_{j}\setminus A_{j-1}, \widehat \a_j(A) >0: \overline Y_{N,j}(A) \leq \be_j(1+\rho) \widehat \a_j(A) \sqrt{N}\Bigg]. 
\eea \]

\begin{lem} \label{miracle_constant} For $N \uparrow \infty$, it holds
\[\bea
p_N^{\de, \rho}\big(j,{\Diamond}\big) = {\mathcal C}_{j, \de} \times 2^{ - G_j N} \int_{\Diamond}  \be_j \exp\left[-\be_j x + o(1) \right] dx + O\big( 2^{-G_j N} e^{-const \times
N}\big)
\eea \]
\end{lem}
\begin{proof}
Clearly, 
\beq \bea 
p_N^{\de, \rho}\big(j,{\Diamond}\big) &= \PP\Big[\overline Y_j \in {\Diamond};\;\forall\; \text{critical}\;B\subsetneq A_{j}\setminus A_{j-1}: \; {Y_{j,B}\over \widehat
\a_j(B)} - {Y_{j, B}^c \over \widehat \a_j^c(B)} \leq -\de \Big] + \\
& \quad - \PP\Big[\overline Y_j \in {\Diamond};\;\exists \; A\subset A_{j}\setminus A_{j-1}, \widehat \a_j(A) >0: \overline Y_{N,j}(A) > \be_j(1+\rho) \widehat \a_j(A) \sqrt{N} \Big] \\
& = (I) - (II).
\eea \eeq
As for $(I)$, we claim that, somewhat surprisingly, the random variable $\overline Y_{j} = \sqrt{N} Y_j -a_{N,j}$ is independent of the
collection \hbox{\small{$\Big( {Y_{j, B}\over \widehat \a_j(B)} - { Y_{j, B}^c\over\widehat \a_j^c(B)}; B\subsetneq A_{j}\setminus A_{j-1} \; \text{is critical} \Big)$}}. This is
best seen by inspection of the covariance: for critical $B$, since $Y_{j} = Y_{j, B} + Y_{j, B}^c$, we have
\[ 
\E\Bigg[ Y_j \cdot \Bigg( {Y_{j, B}\over \widehat \a_j(B)} - { Y_{j, B}^c\over\widehat \a_j^c(B)} \Bigg)\Bigg] = {1\over \widehat \a_j(B)} \E\big[Y_{j, B}^2 \big] - {1\over
\widehat \a_j^c(B)} \E\big[(Y_{j, B}^c)^2 \big] = 0, 
\]
and thus $(I) = {\mathcal C}_{j, \de} \times p_N(j, \Diamond)$ exactly. On the other hand, 
\[0 \leq (II) \leq \sum_{A\subset A_{j}\setminus A_{j-1}, \widehat \a_j(A)> 0} p_N^{>}(j, \Diamond, A, \rho).\]
The Lemma then obviously follows by the asymptotics established in Lemma \ref{quadratic_expansion}.  
\end{proof}

We may now move to the multivariate Poisson approximation of \eqref{multivariate}. First we observe that by Lemma \ref{miracle_constant}, 
\[  
\lim_{N\to \infty} \E\left[ \sum^{(r)} \de_{\overline X_{\s^r, \s(j)}}(B_r) \right] = \lim_{N\to \infty} 2^{G_j N} p_N^{\vare_1, \vare_2}(j, B_r) = \int_{B_r} {\mathcal C}_{j,
\vare_1} \be_j \exp\Big[-\be_j t \Big]dt = \mu_{\vare_1}(B_r). 
\]
According to \cite[p. 236]{barbour}, the multivariate Poisson convergence is equivalent to weak convergence of the sum of the vector's
componenent, $V_N \defi \sum_{r=1}^k \sum^{(r)} \de_{\overline X_{\s(1),\dots, \s(j)}}(B_r)$, towards a Poisson random variable, say $V$, of parameter $\sum_{r=1}^ k
\mu(B_r)$. To see that this is the case, we introduce the index set 
\[\Gamma\defi \left\{ (r, \s^r, \s(j)): r =1,\dots, k,\; \s(j)\in \Sigma_{N,A_j\setminus A_{j-1}}^{R,\vare_1, \vare_2} \right\}.\]
For given $\a = (r, \s^r, \s)\in \Gamma$, consider the subset $\Gamma_\a \subset \Gamma$ consisting of those 
$(q, \s^q, \tau) \in \Gamma$ with the random variables $\overline X_{\s^r, \s}$ and $\overline X_{\s^q, \tau}$ such that $\E\left(\overline X_{\s^r, \s}
\overline X_{\s^q,\tau}\right) \neq a_{N,j}^2$, that is they are correlated. (In the classical Chen-Stein terminology, $\Gamma_\a$ is the "weak dependency
neighborhood" of the index $\a$.)  We set 
\[
p_\a \defi \PP\Big[ \overline X_{\s^r, \s} \in B_r, (\s^r, \s) \;\text{satisfies truncation}\; {\bf T}_1(\vare_1), {\bf T}_2(\vare_2) \Big]
\] 
and define 
$Z_\a \defi \sum_{(q, \s^q, \tau)\in \Gamma_\a}^\star \de_{\overline X_{\s^q, \tau}}(B_q)$, the sum running over those configurations satisfying condition ${\bf
T_1}(\vare_1)$ and ${\bf T_2}(\vare_2)$.  According to the {\it Chen-Stein bound}, cfr. \cite[Theorem 1.A]{barbour}, the total variation distance between $V_N$ and $V$ is bounded above by
\beq \label{chenstein_bound}
\sum_\a \Big\{ p_\a^2 + \sum_{\a' \in \Gamma_\a} p_\a p_{\a'} \Big\} + \sum_{\a = (r, \s^r, \tau)\in \Gamma} \E[\de_{\overline X_{\s^r, \tau}}(B_r) 1_{{\bf T_1, T_2}\;
\text{are satisfied}} \times Z_\a]. 
\eeq
Writing things out, one immediate realizes that exactly the same terms as in Proposition \ref{onset_um_one} make their appearance in expression
\eqref{chenstein_bound}. (These terms are in fact taken care of by Lemma \ref{quadratic_expansion}.) Here is the upshot: the first sum is of order $\exp(- const \times N)$ for some positive $const$, while the second sum is
bounded, {\it mutatis mutandis}, by a constant times the l.h.s of \eqref{upper_bound_dependency}. The total variation distance between $V_N$ and $V$ is therefore of order 
$\exp(- const \times \vare_1\sqrt{N})$. Letting $N\to \infty$ yields the Poisson convergence and settles therefore the proof of Proposition \ref{partial_energies}.
\begin{flushright}
$\square$
\end{flushright}

\section{The Gibbs measure} \label{gibbs_measure}
For $ \be_m < \be <\be_{m+1}$ and $m$ strictly less than $K$, a partial structure only has
emerged. A portion of the system is {\it frozen} and displays hierarchical organization (the collection of points given by $\widehat X_{\s(1),\dots,
\s(m)}, \s \in \Sigma_{N,A_m}$). The portion of the system in {\it
high-temperature} shows no organization at all, and has negligible fluctuations: to be more precise, fix $\s \in \Sigma_{N,A_m}$ and
set
\[ Z_\s \defi \sum_{\tau \in \Sigma_N: \tau_{A_m} = \s} \exp\Bigg[\be\Big(X_{\tau(1),\dots,
\tau(m+1)}+\dots+X_{\tau(1),\dots, \tau(K)} \Big)  \Bigg]. \]
\begin{lem} \label{fluctuations_rem}
Let $\be_m <\be< \be_{m+1}$. There exist constants $\de_1, \de_2 \in (0,1)$ such that 
\[ 
\PP\Bigg[ \Big| \log {Z_\s \over \E[Z_\s]}\Big| \geq N^{-\de_1}\Bigg] \lesssim \exp\left[- N^{\de_2}\right]. 
\] 
\end{lem} 

\begin{proof} This is an adaptation of \cite[Lemma 3.1]{bk2} to the more general setting considered here, so we only sketch the main differences. We first observe that  
\[
\E\big[ Z_{\s}\big] = \exp\left[\sum_{j=m+1}^K {\be^2\over 2} \D_j N + N G_j {\log 2} \right].
\] 
For $A\subset  (I\setminus A_m), \tau \in \Sigma_N$ and $\vare >0$ we set
\small{\[ 
X_{\tau}(A) \defi \sum_{J \in \widehat {\cal P}_{A, m}} X_{\tau_J}^{J},\quad \widehat Z_\s \defi \widehat \sum \exp\Big[\be\big(X_{\tau(1),\dots,
\tau(m+1)}+\dots+X_{\tau(1),\dots, \tau(K)}\big) \Big],  
\]}
where $\widehat \sum$ runs over those $\tau \in \Sigma_N$ such that  $\tau_{A_m} = \s$  and for all $A\subseteq I\setminus A_m$ the
random variables $X_{\tau}(A)$ are bounded by $(\be+\varepsilon) \widehat \a_m(A) N $. We proceed to show that the claim of the Lemma holds, at least for small enough
$\vare$. We first write 
\[
{Z_\s \over \E[Z_\s]} =
\frac{\widehat Z_\s}{\E[\widehat Z_\s]}\times  \frac{\E[\widehat Z_\s]}{\E[Z_\s]} + \frac{Z_\s - \widehat Z_\s}{\E[Z_\s]} = (I)\times
(II)+ (III).
\]
It is easily seen that to $\varepsilon>0$ one can find $\eta>0$ such that $1-e^{-\eta N} \leq (II)\leq 1$, for $N$ large enough. 
This, together with Markov inequality entails that \hbox{\small{$\PP\left[(III) \geq e^{- \eta N/2}\right]\lesssim e^{- \eta N/2}$}}.\\
Therefore,  on a set of $\PP$-probability exponentially close to unity, the following holds: 
\beq \label{bounds_fluctuations}
{Z_\s \over \E[Z_\s]} = (I)\times \left\{1 - O(e^{-const N})\right\} + O\left(e^{-const N}\right), 
\eeq
for $N\to \infty$ and some $const >0$ whose precise value is not important. In particular, we see from \eqref{bounds_fluctuations} that the claim of the Lemma follows as soon as we prove
that for some $\de_1, \de_2 \in (0,1)$ 
\beq  \label{claim_reformulated}
\PP\left[ \big| \log (I)\big| \geq N^{-\de_1}\right] \lesssim \exp\left[- N^{\de_2} \right].
\eeq
To see the latter, let us fix $\de_1 \in (0,1)$. We write: 
\beq \bea \label{rem_condition}
&\PP\left[ \big| \log (I)\big| \geq N^{-\de_1}\right] \\
& \qquad = \PP\Big[ (I) \geq \exp(N^{-\de_1}) \; \text{or} \; (I) \leq \exp(-N^{-\de_1})\Big] \\
& \qquad =  \PP\left[ \Big( (I) - 1\Big)^2 \geq (\exp(N^{-\de_1}) - 1)^2 \; \text{or}\; \Big( (I) - 1\Big)^2 \geq (\exp(-N^{-\de_1}) - 1 )^2\right] \\
& \qquad \leq \PP\left[ \Big( (I) - 1\Big)^2 \geq \min\left\{ (\exp(N^{-\de_1}) - 1)^2; (\exp(-N^{-\de_1}) - 1 )^2\right\} \right] \\
& \qquad \stackrel{\text{(Markov)}}{\leq} {1\over m(N, \de_1)} {\E\Big[ (\widehat Z_\s -
\E[\widehat Z_\s]\big)^2\Big]\over \E\big[\widehat Z_\s\big]^{2} }, 
\eea \eeq
with $m(N, \de_1) \defi \min\left\{ (\exp(N^{-\de_1}) - 1)^2; (\exp(-N^{-\de_1}) - 1 )^2\right\}$. It is now crucial that $\be < \be_{m+1}$ strictly: this
ensures that for $\varepsilon$ small enough (recall the construction of the chain ${\bf T}$) we have  
\beq \label{rem_phase}
\eta' \defi \inf_{A\subset (I\setminus A_m)} \left\{\g(A) \log 2 - \left[\be^2 - {\big(\be-\varepsilon\big)^2\over 2} \right] \widehat \a_m(A) \right\} >0.
\eeq
Given this, expanding the square in the numerator of the r.h.s of \eqref{rem_condition} and exploiting the usual bounds on gaussian integrals yields 
\beq \bea 
\PP\Big[\big|\log(I)\big| \geq N^{-\de_1} \Big] &\lesssim {1\over m(N, \de_1)} \sum_{A\subset (I\setminus A_m)}  2^{-\g(A)N} \exp\Big[ N\big(\be^2 - {(\be -\varepsilon)^2\over 2}\big) 
\widehat\a_m(A) \Big] \\
& \stackrel{\eqref{rem_phase}}{\lesssim} {\exp\big[- \eta' N \big] \over m(N, \de_1)}, 
\eea \eeq
which is clearly more than needed to get \eqref{claim_reformulated}. Lemma \ref{fluctuations_rem} then easily follows. 
\end{proof}

\begin{lem} \label{summability}
Let $\epsilon >0$. There exists positive $\phi$ such that
\beq \label{summabilityone}
\PP\left[ \sum_{\exists j \leq m: \widehat X_{\s(1),\dots, \s(j)} \leq -\phi } \exp\Big[\be\big(X_{\s}-a^m_N\big)\Big] \geq \epsilon \right] \leq \epsilon.
\eeq
\end{lem}
\begin{proof}
By Proposition \ref{localized_r} we can find $C>0$ such that (for large enough $N$)
\[
\PP\left[\forall j\leq m, \forall \tau \in \Sigma_{N,A_j}\; \widehat X_{\tau(1),\dots, \tau(j)} \leq C\right] \geq 1- \epsilon/2,
\]  
in which case the l.h.s of \eqref{summabilityone} is then bounded by \hbox{\small{$\PP\Big[ \widehat \sum \exp\big[\be(X_{\s}-a^m_N)\big] \geq \epsilon \Big]
+\epsilon/2,$}} with $\widehat \sum$ running over those $\s \in \Sigma_{N}$ such that $\widehat X_{\s(1),\dots, \s(l)}\leq C$ for all $l = 1,\dots, m$ but
$\widehat X_{\s(1),\dots, \s(j)} \leq -\phi$  for some $j=1,\dots, m$. We have:
 \beq \bea \label{summability_two}
&\PP\Big[ \widehat \sum \exp\big[\be(X_{\s}-a^m_N)\big] \geq \epsilon \Big] \leq  \\
&\leq {\epsilon^{-1}} \mathop{\sum_{\s \in \Sigma_N}}_{j=1,\dots, m} \E\Big[
\exp\big[\be(X_\s-a^m_N)\big]; \forall l\leq m: \widehat X_{\s(1),\dots, \s(l)} \leq C,\; \widehat
X_{\s(1),\dots, \s(j)} \leq -\phi \Big]\\
& \leq \epsilon^{-1}2^{\g(A_m) N} \sum_{j=1}^{m}\E\Big[\exp\big[\be\widehat Y_m\big];\; \forall l\leq m: \widehat Y_l
\leq C,\; \text{but}\; \widehat Y_j \leq -\phi \Big]\\
& \lesssim  \epsilon^{-1} \sum_{j\leq m} \exp\Big[ \sum_{l\neq j} (\be_{l+1}-\be_l)C - \big(\be_{j+1}-\be_j\big)\phi + o(1)\Big]
\eea \eeq
(the first step above by Markov inequality, the second by simply integrating out the unrestricted random variables $X_{\s(1),\dots, \s(l)}$ for $l=m+1,\dots, K$, and the
third by Lemma \ref{control}). It thus suffices to choose $\phi$ large enough in the positive to have $\eqref{summability_two}\leq \epsilon/2$.  
\end{proof}

\begin{prop} \label{localization}
Let $\epsilon >0$ and $\be \in (\be_{m}, \be_{m+1})$ $(m= 1, \dots, K)$. There exists $C >0$ such that 
\[\PP\left[ {\mathcal G}_{\be, N}\left(\exists j\leq m:\; \overline X_{\s(1),\dots, \s(j)} \notin \left[-C, C\right]\right) \geq \epsilon\right] \leq \epsilon\] 
for large enough $N$. 
\end{prop}
\begin{proof}
We will prove that to arbitrary $\epsilon >0$ there exists $\widehat C>0$ such that 
\beq \label{variant}
\PP\left[ {\mathcal G}_{\be, N}\left(\exists j\leq m:\;
\widehat X_{\s(1),\dots, \s(j)} \notin \left[-\widehat C,\widehat C\right]\right) \geq
\epsilon\right] \leq\epsilon,
\eeq
as this obviously implies that there exist $\overline C>0$ such that the claim of Proposition \ref{localization} holds. To see
\eqref{variant}, we first modify the definition of the Gibbs measure slightly, subtracting the
constant $\be a_N$ to the energies: \hbox{\small{$ {\mathcal G}_{\be, N}(\s) = \exp\big[ \be(X_\s- a^m_N)\big]\big/ Z_{a^m_N}(\be)$}}
with \hbox{\small{$Z_{a^m_N}(\be) \defi \sum_{\tau \in \Sigma_N} \exp\big[ \be(X_\tau- a^m_N)\big]$}}.

We now claim that to given $\epsilon$ there exists $\eta>0$ such that, for $N$ large enough 
\beq \label{localization_one}
\PP\Big[Z_{a^m_N}(\be) \leq \eta \Big] \leq {\epsilon \over 2}.
\eeq 
The l.h.s above is to any $R>0$ evidently bounded by 
\[
\PP\Bigg[\widehat\sum_R \exp\Big[\be\big(\widehat
X_{\s(1),\dots, \s(m)} + {1\over \be} \log {Z_{\s(1),\dots, \s(m)} \over \E[Z_{\s(1),\dots, \s(m)}]}\big) \Big] \leq \eta \Bigg]
\]
with $\widehat \sum_R$ running over those $\s\in \Sigma_{N,A_m}$ only such that $\widehat X_{\s(1),\dots, \s(m)} \in (-R,R)$. It is also
easily seen that to any $\epsilon'>0$ this sum runs over at most ${\sf N}={\sf N}(\epsilon')$ configurations with $\PP$-probability greater than $(1-\epsilon')$. By Lemma \ref{fluctuations_rem}
the contributions of each term $\log \big(Z_\s/\E[Z_\s]\big)$ associated to these $\sf N$ configurations is in the large $N$ limit irrelevant. It is therefore sufficient to prove that to $\tilde \epsilon$ there exist $\tilde \eta$ such that
\[ 
\PP\left[\widehat \sum_R \exp\Big[\be \widehat
X_{\s(1),\dots, \s(m)} \Big] \leq \tilde \eta \right] \leq {\tilde \epsilon \over 2}.
\]
This is however straightforward, since for $x < -R$    
\beq 
\widehat \sum_R \exp\Big[\be \widehat
X_{\s(1),\dots, \s(m)} \Big] \leq  \exp(\be x) \Longrightarrow  \sharp\Big\{\s\in \Sigma_{N, A_m}:\; \widehat X_{\s(1),\dots, \s(m)} \geq -R\Big\} = 0.
\eeq
By Proposition \ref{partial_energies} and the properties of the limiting process $\widehat X_m$, it is easily seen that the probability of
the event on the r.h.s above can be made (for large enough $N$) as small as needed by simply choosing $R$ large enough in the positive.
On the other hand,  by Proposition \ref{existence_lemma} and Lemma \ref{summability}, to given $\eta, \epsilon>0$ we can find positive $\widehat C$ such that 
\[
\PP\left[ \sum_{ \s\in \Sigma_{N};\; \exists j\leq m:\; \widehat X_{\s(1),\dots, \s(j)}\notin \left[-\widehat C,\widehat C\right]}\exp\Big[\be(X_\s-a^m_N)
\Big] \geq \eta \epsilon\right] \leq {\epsilon \over 2},
\]
which together with \eqref{localization_one} yields  \eqref{variant} and thus settles the proof of Proposition \ref{localization}.
\end{proof}

\section{Proof of Theorem \ref{main_theorem}} \label{proof_main_theorem}

\noi{\bf The Gibbs measure, $\boldsymbol{\be> \be_K}$.} Recall that $\Xi_{\be, N}$ is the law on \hbox{\small{${\mathcal M}_{mp}\left( (\R^+)^{(2)} \times 2^{I} \right)$}}  naturally induced by the points
\hbox{\small{$(\exp[\be(X_{\s} -a_N)]\big/ Z_{a_N}(\be), \s\in \Sigma_{N})$}}. 

Set $H_{N,K} \defi \left( \exp\left[\be(X_\s -a_N)\right], \s\in \Sigma_N \right)$.
This is nothing else than the image of the PP of the energy levels under the mapping $\exp(\be \cdot)$, in which case {\it (cfr. \cite[Prop.
8.5]{boszni} and a straightforward generalization)} it follows by Proposition \ref{partial_energies} that $H_{N,K}$ converges weakly to a PP $H_K \defi (\eta_{\bf i}, {\bf i} \in \N^K)$ with $\eta_{\bf i} = \eta^1_{{\bf i}_1}\eta^2_{{\bf i}_2}\cdots \eta^K_{{\bf
i}_K}$ and the following properties:  For $l\leq K$ and multi-index ${\bf i}_{l-1}$, the point process \hbox{\small{$(\eta_{{\bf i}_{l-1},
i_l}^l;\; i_l\in \N)$}} is poissonian with density ${\mathcal C}_l x_l(\be) \cdot t^{-x_l(\be)-1} dt$ on $\R^+$;  The $\eta^l$ are independent for different $l$; $(\eta_{{\bf i}_{l-1}, i_l}^l;\; i_l \in \N)$ are independent for different ${\bf i}_{l-1}$.
Given such a PP, it is easily seen that $\sum_{{\bf i}} \eta_{\bf i} <\infty$ almost surely. (This is mainly due to
the fact that $x_1(\be) < x_2(\be) < \dots < x_K(\be)$. For more on this, cfr. \cite[Prop. 9.5]{boszni}
and a straightforward generalization.) We may thus consider the new collection of {\it normalized} points given by $(\overline {\eta_{\bf i}}; {\bf i}\in
\N^K)$, which induces naturally an element of \hbox{\small{${\mathcal
M}_{mp}\left( (\R^+)^{(2)} \times 2^I\right)$}} with possible marks those from the chain ${\bf T} = \{A_0, A_1, \dots, A_K\}$ only. We denote by $\Xi_{\be}$ its law. \\

\noi With the new notation $Z_{a_N}(\be) = \int x H_{N,K}(dx)$, and by Proposition \ref{existence_lemma} and Lemma \ref{summability} we have
that to $\epsilon>0$ there exists $C>0$ such that 
\[
\PP\Bigg[ \int_0^{1/C} x H_{N,K}(dx)+ \int_{C}^{\infty} x H_{N,K}(dx) \geq \epsilon \Bigg]\leq \epsilon,
\]
for large enough $N$. This implies that by uniformly approximating $f(x)=x$ through continous functions of the form
\small{\begin{equation} 
\tilde f(x)=
\begin{cases}
x, \quad &x\in [1/C, C]\\ 
0,\quad &x\notin [1/ 2C, 2 C]\\
\end{cases}
\quad \text{and}\quad \tilde f(x)\leq x,\; \forall x\in \R_+,
\end{equation}}
we have weak convergence of $Z_{a_N}(\be)$ to 
$\int x H_{K}(dx) = \sum_{\bf i} \eta_{\bf i}$. But by continuity of the mapping
\[ \bea
{\mathcal M}_{mp}\left( (\R^+)^{(2)} \times 2^{I} \right)\times (0, \infty) &\to {\mathcal M}_{mp}\left( (\R^+)^{(2)} \times 2^{I}\right) \\
\left(\sum_{i} \de_{\{y_i;\, f_i\}}, \, A\right) &\mapsto  \sum_i \de_{\{{y_i/A}; \; f_i\}}
\eea\]
and Proposition \ref{partial_energies}, we then also have that $\Xi_{N, \be}$ converges weakly to $\Xi_{\be}$. \\

It is not difficult to see that the laws $\Xi_{\be}$ and $P_{x_K} \sqcap Q_{{\bf T}, {\bf t}}$ coincide (this easily follows from the way the coalescent \cite{boszni} is constructed). This settles the proof of the first claim. \\

\noi{\bf The marginal, $\boldsymbol{\be> \be_m}$.} For convenience, we assume that $\be \in
(\be_k, \be_{k+1})$ for some $k\geq m$ and regard ${\mathcal G}_{\be, N}^{(m)}$ as a marginal of ${\mathcal G}_{\be, N}^{(k)}$: for $\s\in \Sigma_{N,A_m}$ we write 
\[
{\mathcal G}_{\be,N}^{(m)}(\s) = \sum_{\tau \in \Sigma_{N}: \tau_{A_m} = \s} \exp\Big[\be \widehat X_{\tau(1),\dots,
\tau(k)} +  \log {Z_{\tau(1),\dots,\tau(k)}\over \E[\tau(1),\dots,\tau(k)]} \Big] \Big/ Z_{a_N}(\be). 
\]

We now claim that the weak limit of ${\mathcal G}_{\be,N}^{(m)}$ coincides with that of the process naturally induced by the points 
\[
\widehat {\mathcal G}^{(m)}_{\be,N}(\s) \defi \mathop{\sum_{\tau \in \Sigma_{N, A_k},}}_{\tau_{A_m = \s} }{ \exp\big[\be \widehat X_{\tau(1),\dots, \tau(k)}\big] \over
\widehat Z_{m}(\be)}, \quad \widehat Z_{m}(\be) \defi \sum_{\eta\in \Sigma_{N,A_k}} \exp\big[ \be \widehat X_{\eta(1),\dots, \eta(k)}\big].
\]
In fact, by Proposition \ref{localization}, to given $\epsilon>0$ there exists $C>0$ such that 
\[ 
\PP\Big[ {\mathcal G}_{\be, N}^{(k)}\Big( \s \in \Sigma_N:\; \overline X_{\s(1),\dots, \s(l)}\in [-C, C]\; \forall l\leq k \Big) \geq 1-\epsilon \Big] \geq
1-\epsilon,
\] 
for large enough $N$. Moreover, there exists ${\sf N}={\sf N}(\epsilon)$ such that $\PP\big[ \sharp\{\Sigma_{N,
A_k}^{C} \}\geq {\sf N} \big] \leq \epsilon$, and by Lemma \ref{fluctuations_rem} the
fluctuations of these $\sf N$ r.v.'s \hbox{\small{$\log Z_{\tau(1),\dots, \tau(k)}/ \E[Z_{\tau(1),\dots, \tau(k)}]$}} are negligible. \\

Therefore, the weak limit of ${\mathcal G}_{\be,N}^{(m)}$ and ${\widehat{\mathcal G}}_{\be,N}^{(m)}$ coincide. \\

\noi We rewrite the points as 
\[\bea 
&\hspace{2cm} \widehat {\mathcal G}^{(m)}_{\be, N}(\s) = {\exp \be \big[\widehat X_{\s(1),\dots, \s(m)}+ U_{\s(1),\dots, \s(m)} \big]\over  \widehat Z_m(\be)}\\
& U_{\s(1),\dots, \s(m)} = {1/\be} \log \mathop{\sum_{\tau \in \Sigma_{N, A_k},}}_{\tau_{A_m} = \s} \exp \be\Big[ \overline X_{\tau(1),\dots,
\tau(m+1)}+\dots \overline X_{\tau(1),\dots, \tau(k)} \Big].
\eea \]
To fixed $\s \in \Sigma_{N, A_m},\; U_{\s} = U_{\s(1),\dots, \s(m)}$ is (up to a constant) the logarithm of the {\it partition function} of an irreducible
hamiltonian in low temperature ($\be > \be_m$). A fixed realization $(\widehat X_{\s(1),\dots, \s(m)} + U_{\s(1),\dots, \s(m)}; \s\in \Sigma_{N,A_m})$ induces naturally
an element of \hbox{\small{${\mathcal M}_{mp}\big( {\R}^{(2)} \times 2^{A_m}\big)$}}, whose law is denoted $\widehat {XU}_{N,m}$. By
Proposition \ref{partial_energies}, and the considerations in the proof of claim a) it is thus easily seen that that $\widehat {XU}_{N,m}$ converges
weakly to the law $\widehat {XU}_m$ of the process on \hbox{\small{${\mathcal M}_{mp}\big( {\R}^{(2)} \times 2^{A_m}\big)$}} (with the possible marks being those from the
restricted chain ${\bf T}^{(m)} = \{A_0, \dots, A_m\}$ only) induced by the collection of points given by $( u _{\bf i} + U_{\bf i}; {\bf i}\in \N^m)$ where 
\[ 
u_{\bf i} \defi u_{{\bf i}_1}^1 + \dots + u_{{\bf i}_{m}},\; U_{\bf i} \defi {1\over \be} \log \sum_{i_{m+1},\dots, i_{k}} \exp\Big[\be\big(u_{{\bf i}_m, i_{m+1}}^{m+1}+\dots+
u^k_{{\bf i}_m, i_{m+1},\dots, i_k} \big) \Big].
\]
For $l= 1,\dots, k$ and any multi-index ${\bf i}_{l-1}$ the point process \hbox{\small{$(u_{{\bf i}_{l-1}, i_l}^l;\; i_l\in \N)$}}
is poissonian with density \hbox{\small{${\mathcal C}_l \be_l \exp(-\be_l t)dt$}}. The $u^l$ are independent for different $l$ and \hbox{\small{$(
u_{{\bf i}_{l-1}, i_l}^l;\; i_l \in \N)$}} are independent for different ${\bf i}_{l-1}$. An important observation is that to fixed ${\bf i}_{m-1}$ the
PP $(u^m_{{\bf i}_{m-1}, i_m} + U_{{\bf i}_{m-1}, i_m};\; i_m\in \N)$ is simply a shift by independent
variables of a PPP, in which case it is easy to see that  
\beq \label{equality_distribution} 
\Big(u^m_{{\bf i}_{m-1}, i_m} + U_{{\bf i}_{m-1}, i_{m}} - const;\; i_m\in \N\Big) \stackrel{\text{(distr)}}{=} \Big(u^m_{{\bf i}_{m-1}, i_m};\; i_m\in \N\Big),
\eeq 
for some {\it const $>0$}, cfr. \cite[Prop. 8.7]{boszni} and a straightforward generalization. By continuity under mappings, the process on \hbox{\small{${\mathcal
M}_{mp}\left((\R^+)^{(2)} \times 2^{A_m} \right)$}} induced by the points \hbox{\small{$\left(\exp\be \big[\widehat X_{\s(1),\dots, \s(m)}+ U_{\s(1),\dots, \s(m)}-const\big]; \;
\s\in \Sigma_{N, A_m}\right)$}} converges weakly to the process induced by the points $\big(\exp[\be u _{\bf i}];\; {\bf i}\in \N^m\big)$. To get the weak limit
of $\Xi_{\be, N}^{(m)}$ it then suffices to prove that the normalization procedure commutes with the limit $N\to \infty$; this is done exactly as in case a); the proof of the Main Theorem is completed. 
\begin{flushright}
$\square$
\end{flushright}

{\large \bf Acknowledgments.} The research of E.B. was partly supported by a grant of the Swiss National Foundation under Contract No. 200020-116348. Part of the research of N.K. was carried out while visiting the ENS Lyon as a Postdoctoral Fellow: many thanks are due to the
members of the Mathematical Department U.M.P.A. for the friendly and inspiring athmosphere. N.K. also gratefully acknowledges financial support of the Swiss
National Foundation under Contract No. PBZH22-118826, and of the Deutsche Forschungsgemeinschaft under Contract no. DFG GZ BO 962/5-3.

\end{document}